\documentclass[intlimits]{article}

\usepackage{latexsym,amsfonts,amssymb,amsmath,amsthm,amscd}
\usepackage{soul}  % enables line breaks in underlined text
\usepackage[dvipsnames]{xcolor}
\usepackage{comment}
\usepackage{enumitem}
\usepackage{bbm}

\theoremstyle{plain}
\newtheorem{theorem}{Theorem}[section]
\newtheorem{lemma}{Lemma}[section]
\newtheorem{proposition}{Proposition}[section]
\newtheorem{corollary}{Corollary}[section]

\newtheorem{definition}{Definition}[section]

\newtheorem{remark}{Remark}[section]

\numberwithin{equation}{section}
\numberwithin{definition}{section}

\newtheorem*{theorem*}{Theorem}

\def\p{\partial}

\newcommand\JM{Mierczy\'nski}
\newcommand\RR{\ensuremath{\mathbb{R}}}
\newcommand\QQ{\ensuremath{\mathbb{Q}}}

\newcommand\NN{\ensuremath{\mathbb{N}}}

\newcommand{\norm}[1]{\ensuremath{\lVert#1\rVert}}

\DeclareMathOperator{\conv}{\overline{co}}

\DeclareMathOperator*{\esssup}{ess\,sup}
\newcommand{\n}[1]{(#1)}
\newcommand{\nb}[1]{\big(#1\big)}

\newcommand{\nbbb}[1]{\bigg(#1\bigg)}

\newcommand{\U}{U}
\newcommand{\dd}{\mathrm{d}}

\newcommand{\WST}{weak\nobreakdash-\hspace{0pt}\textup{*} }

\definecolor{R}{RGB}{255, 0, 0}
\definecolor{C}{RGB}{255, 0, 238}

\usepackage{hyperref}

\hypersetup{
    colorlinks,
    linkcolor=C,
    citecolor=R,
    urlcolor=teal
}

\usepackage{enumitem}
\usepackage{mathtools} % \coloneqq equal form definition := %

\usepackage{latexsym,amsfonts,amssymb,amsmath,amsthm,amscd}
\usepackage{soul}  % enables line breaks in underlined text

\usepackage[normalem]{ulem}

\usepackage{enumitem}

\usepackage{comment}
\usepackage[dvipsnames]{xcolor}

\usepackage{bbm}
\usepackage{hyperref}

\DeclareFontFamily{U}{mathx}{\hyphenchar\font45}
\DeclareFontShape{U}{mathx}{m}{n}{
      <5> <6> <7> <8> <9> <10>
      <10.95> <12> <14.4> <17.28> <20.74> <24.88>
      mathx10
      }{}
\DeclareSymbolFont{mathx}{U}{mathx}{m}{n}
\DeclareFontSubstitution{U}{mathx}{m}{n}
\DeclareMathAccent{\widecheck}{0}{mathx}{"71}
\DeclareMathAccent{\wideparen}{0}{mathx}{"75}

\begin{document}

\title {Parabolic Differential Equations with Bounded Delay}

\author{\ Marek Kryspin   \ \&  Janusz Mierczy\'nski
\\
Faculty of Pure and Applied Mathematics \\
Wroc{\l}aw University of Science and Technology \\
Wybrze\.ze Wyspia\'nskiego 27 \\
PL-50-370 Wroc{\l}aw, Poland \\
}
\date{} \maketitle

\setcounter{section}{-1}
\maketitle

\begin{abstract}
We show the continuous dependence of solutions of linear nonautonomous second order parabolic partial differential equations (PDEs) with bounded delay on coefficients and delay.  The assumptions are very weak: only convergence in the \WST topology of delay coefficients is required.  The results are important in the applications of the theory of Lyapunov exponents to the investigation of PDEs with delay.
\end{abstract}

\setcounter{section}{-1}
\section{Introduction}
The purpose of the present paper is to formulate and prove results on existence and continuous dependence on parameters of solutions of linear second order partial differential equations (PDEs) of parabolic type with bounded time delay.  To be more specific, consider a rather simplified example, that is, an equation of the form
\begin{equation}
\label{eq:intro-0}
    \begin{cases}
      \displaystyle \frac{\partial u}{\partial t}(t, x) = {\Delta}u(t, x) + c_1(t, x) u(t - R(t),x), & t \in [0, T], \ x \in D
      \\[1.5ex]
      u(t, x) = 0 & t \in [0, T], \ x \in \partial D,
    \end{cases}
\end{equation}
where $D \subset \RR^N$ is a bounded domain with boundary $\partial D$, $\Delta$ is the Laplace operator in $x$, $T > 0$, $c_1 \colon (0, T) \times D \to \RR$ belongs to $L_{\infty}((0, T) \times D)$, and $R \colon [0, T] \to [0, 1]$ is a function in $L_{\infty}((0, T))$.

\medskip
The theory of Lyapunov exponents (or rather, more generally, the theory of skew\nobreakdash-\hspace{0pt}product dynamical systems) is a powerful tool in the applications of the theory of dynamical systems to the investigation of evolution equations (in a broad sense, containing but not excluded to, ordinary differential equations, parabolic partial differential equations, hyperbolic partial differential equations).  That theory requires the (linear) equation to generate a skew\nobreakdash-\hspace{0pt}product dynamical system on some bundle whose base is the closure of the set of coefficients of the original equation.

Let us consider two cases. We will remain in the simplified framework of~\eqref{eq:intro-0}.
\begin{itemize}[label=$\bullet$]
    \item
    The \emph{nonautonomous} case,
    \begin{equation}
    \label{eq:intro-nonauton}
    \begin{cases}
        \displaystyle \frac{\partial u}{\partial t}(t, x) = {\Delta}u(t, x) + c_1(t, x) u(t - R(t),x), & t > 0, \ x \in D
        \\[1.5ex]
        u(t, x) = 0 & t > 0, \ x \in \partial D,
    \end{cases}
    \end{equation}
    where $c_1$ is defined on $(- \infty, \infty) \times D$ and $R$ is defined on $(- \infty, \infty)$.  We take the closure, in an appropriate topology, of the set of all time\nobreakdash-\hspace{0pt}translates of $c_1$ (the so\nobreakdash-\hspace{0pt}called \emph{hull}). The topology must be, on the one hand, coarse enough for the hull to be a compact (metrizable) space, and, on the other hand, fine enough for, first, the time translation operator on the hull to be continuous, and, second, the solution operator to depend continuously on parameters, that is, members of the hull.  The paper~\cite{Sacker} gives a survey of subsets of function spaces that can serve as hulls.

    For the theory of linear skew\nobreakdash-\hspace{0pt}product (semi)flows on bundles whose fibers are Banach spaces and some of its applications, see, e.g., \cite{Sa-Se}, \cite{Chow-Le95}, \cite{Chow-Le96}, \cite{Shen-Yi}, \cite{Chi-Lat}, \cite{Monogr}, \cite{MiShPart3}, \cite{Ba-Dr-Va} for a very incomplete list arranged in chronological order.
    \item
    The \emph{random} case,
    \begin{equation}
    \label{eq:intro-random}
    \begin{cases}
        \displaystyle \frac{\partial u}{\partial t}(t, x) = {\Delta}u(t, x) + c_1(\theta_{t}\omega, x) u(t - R(\theta_{t}\omega),x), & t > 0, \ x \in D
        \\[1.5ex]
        u(t, x) = 0 & t > 0, \ x \in \partial D,
    \end{cases}
    \end{equation}
    where $c_1$ is now defined on $\Omega \times D$ and $R$ is defined on $\Omega$, with $(\Omega, \mathfrak{F}, \mathbb{P})$ a probability space on which an ergodic measurable flow $\theta = (\theta_t)_{t \in \RR}$ acts.  Here the role of hull is played by $\Omega$, and the measurability of the flow $\theta$ is one of the assumptions.  In order to apply the theory of Lyapunov exponents in the measurable setting, as presented in, e.g., \cite{Lian-Lu}, \cite{Fro-Ll-Qu}, \cite{GT-Qua14}, \cite{GT-Qua15}, \cite{Bl16}, one needs to show the measurable dependence of the solution operators on $\omega \in \Omega$
\end{itemize}
In the present paper we address the problem of continuous dependence on members of the hull.  As the space of coefficients we take a closed and bounded subset of the Banach space of essentially bounded (Lebesgue\nobreakdash-\hspace{0pt})measurable functions on $(0, T) \times D$, where $T > 0$, with the \WST topology induced by the duality pairing between $L_1$ and $L_{\infty}$.   Regarding the zero order coefficients and delay terms, no additional assumption is made.  In~particular, the dependence on $t$ can be quite weak.

Although there have been a lot of papers dealing with the issues of the existence of solutions of delay PDEs (many of them nonlinear, and admitting more general delay terms, employing various definitions of solutions, see, e.g., \cite{Tr-Webb74}, \cite{Tr-Webb76}, \cite{Tr-Webb78}, \cite{Fitz77}, \cite{Fitz78}, \cite{Ma-Sm90}, \cite{Ma-Sm91}, \cite{Wu}, \cite{Fr03}, \cite{Fr-Ni03}, \cite{Ba-Pia01}, \cite{Ba-Pia02}, \cite{Ba-Pia06}), the only papers we are aware of dealing explicitly with continuous dependence of solutions of delay PDEs on parameters are \cite{No-Nu-Ob-Sa} and \cite{Ob-Sa}.

\medskip
To give a flavor of our results, we formulate now some specializations of our main results to the case of~\eqref{eq:intro-0}.  We assume $1 < p < \infty$.

The first, a specialization of Theorem~\ref{thm:mild_solution_existence},  establishes the existence and uniqueness of mild solutions.
\begin{theorem*}
    Let $c_1 \in L_{\infty}((0, T) \times D)$, $u_0 \in C([-1, 0], L_p(D))$ and $R \in L_{\infty}((0, T))$ be such that $R(t) \in [0, 1]$ for Lebesgue\nobreakdash-\hspace{0pt}a.e.\ $t \in (0, T)$.  Then there exists a unique solution $u(\cdot; c_1, u_0, R) \in C([-1, T])$ of Eq.~\eqref{eq:intro-0} with initial condition $u(t; c_1, u_0, R) = u_0$, $t \in [-1, 0]$.  The solution is understood in a suitable integral sense \textup{(}a mild solution\textup{)}.
\end{theorem*}

The second, a specialization of Theorem~\ref{th:joint_continouity}\ref{th:joint_continouity_ii}, establishes the continuity, in a suitable sense, of a solution with respect to initial conditions and parameters.
\begin{theorem*}
    Assume that $(c_{1,m})_{m = 1}^{\infty}$, $(u_{0,m})_{m = 1}^{\infty}$ and $(R_m)_{m = 1}^{\infty}$ are sequences satisfying the following:
    \begin{itemize}[label=$\bullet$]
        \item
        $c_{1,m}$ have their $L_{\infty}((0, T) \times D)$\nobreakdash-\hspace{0pt}norms uniformly bounded, and converge in the \WST topology to $c_1 \in L_{\infty}((0, T) \times D)$;
        \item
        $u_{0,m}$ converge in the norm topology of $C([-1, 0], L_p(D)$ to $u_0$;
        \item
        $R_m$ converge for Lebesgue\nobreakdash-\hspace{0pt}a.e.\ $t \in (0, T)$ to $R$.
    \end{itemize}
    Then
    \begin{equation*}
        u(\cdot; c_{1,m}, u_{0,m}, R_m) \to u(\cdot; c_{1}, u_{0}, R)
    \end{equation*}
    in the $ C([-1, 0], L_p(D))$\nobreakdash-\hspace{0pt}norm.
\end{theorem*}

\medskip
The paper is organized as follows.

Section~\ref{sect:assumptions} presents the assumptions used throughout.

In Section~\ref{sect:weak-sol} results concerning the existence and basic properties of (weak) solutions to linear parabolic PDEs without delay terms are gathered.  They are for the most part taken from~\cite{Monogr} and based on~\cite{Dan2}, though some of them (Proposition~\ref{prop:higher_orders-skew_product-norm_continuity}, for example), perhaps belonging to the folk lore,  appear in~print for the first time.

Section~\ref{sect:mild} is devoted to defining and proving the existence and uniqueness of (mild) solutions of PDEs with delay terms.  Section~\ref{sect:dependence_on_IC} provides estimates of the solutions which are then used to prove the continuous dependence on initial conditions.

Section~\ref{sect:dependence-on-parameters} can be considered the main part of the paper.  Here the continuous dependence of solutions on coefficients and delay terms is proved under very weak assumptions: coefficients are required to converge in the \WST topology only.

\medskip
It should be mentioned that a similar approach has been successfully applied in the case of ordinary differential equations with delay in~\cite{MiShPart3}, \cite{Mi-No-Ob18}, \cite{Mi-No-Ob20}, see also \cite[Chpt.~5]{Do-diss}, \cite{Do-Sieg}, \cite{Ba-Va}.

\subsection{General Notations}
We write $\RR^{+}$ for $[0, \infty)$, and $\QQ$ for the set of all rationals.

If $B \subset A$, we write $\mathbbm{1}_{\! B}$ for the \emph{indicator} of $B$: $\mathbbm{1}_{\! B}(a) = 1$ if $a \in B$ and $\mathbbm{1}_{\! B}(a) = 0$ if $a \in A \setminus B$.

For a metric space $(Y, d)$, $\mathfrak{B}(Y)$ denotes the
$\sigma$\nobreakdash-\hspace{0pt}algebra of all Borel subsets of $Y$.

For Banach spaces $X_1$, $X_2$ with norms $\lVert \cdot \rVert_{X_1}$, $\lVert \cdot \rVert_{X_2}$, we let $\mathcal{L}(X_1, X_2)$ stand for the Banach space of bounded linear mappings from $X_1$ into $X_2$, endowed with the standard norm $\lVert \cdot \rVert_{X_1, X_2}$.  Instead~of $\mathcal{L}(X, X)$ we write $\mathcal{L}(X)$, and instead~of $\lVert \cdot \rVert_{X, X}$ we write $\lVert \cdot \rVert_{X}$.  $\mathcal{L}_{\mathrm{s}}(X_1, X_2)$ denotes the space of bounded linear mappings from $X_1$ into $X_2$ equipped with the strong operator topology. Instead~of $\mathcal{L}_{\mathrm{s}}(X, X)$ we write $\mathcal{L}_{\mathrm{s}}(X)$.

\smallskip
Throughout the paper, $T > 0$ will be fixed.

We set
\begin{equation*}
   \Delta \vcentcolon= \{\, (s, t) \in \RR^2 : 0\le s \le t\le T \,\}, \qquad \dot{\Delta} \vcentcolon= \{\, (s, t) \in \RR^2 : 0\le s < t\le T \,\}.
\end{equation*}

Throughout the paper, $D \subset \RR^N$ stands for a bounded domain, with boundary $\partial D$.

By $\mathfrak{L}((0, T))$ we understand the $\sigma$\nobreakdash-\hspace{0pt}algebra of all Lebesgue\nobreakdash-\hspace{0pt}measurable subsets of $(0, T)$.  The notations $\mathfrak{L}(D)$ and $\mathfrak{L}((0, T) \times D)$ are defined in a similar way.

For $u$ belonging to a Banach space of (equivalence classes of) functions defined on $D$ we will denote by $u[x]$ the value of $u$ at $x \in D$.

$L_p(D)=L_p(D,\RR)$ has the standard meaning, with the norm, for $1 \le p < \infty$, given by

\begin{equation*}
    \lVert u \rVert_{L_p(D)} \vcentcolon= \biggr( \int_{D} \lvert u[x] \rvert^{p} \, \mathrm{d}x \biggr)^{\!\frac{1}{p}},
\end{equation*}
and for $p = \infty$ given by
\begin{equation*}
    \lVert u \rVert_{L_{\infty}(D)} \vcentcolon=  \esssup\limits_{D}{u}.
\end{equation*}

For $1 \le p \le \infty$ let $p'$ stand for the H\"older conjugate of $p$.  The duality pairing between $L_p(D)$ and $L_{p'}(D)$ is given, for $1 < p < \infty$, or for $p =1$ and $p' = \infty$, by
\begin{equation*}
        \langle u, v \rangle_{L_p(D), L_{p'}(D)} =  \int\limits_{D} u[x] \, v[x] \, \mathrm{d}x, \quad
    u \in L_p(D), \ v \in L_{p'}(D).
\end{equation*}

Let $u$ be an equivalence class of functions defined for Lebesgue\nobreakdash-\hspace{0pt}a.e.\ $t \in (0, T)$ and taking values in $L_p(D)$, $1 \le p < \infty$ (in the sequel we will refer to such $u$ simply as a function).

\begin{itemize}[label=$\bullet$]
    \item
    $u$ is said to be \emph{measurable} if it is $(\mathfrak{L}((0, T)), \mathfrak{B}(L_p(D)))$\nobreakdash-\hspace{0pt}measurable, meaning that the preimage under $u$ of any open subset of $L_p(D)$ belongs to $\mathfrak{L}((0, T))$.
    \item
    $u$ is \emph{strongly measurable} (sometimes called \emph{Bochner measurable}) if there exists a sequence $(u_m)_{m = 1}^{\infty}$ of simple functions such that $\lim\limits_{m \to \infty} \lVert u_m(t) - u(t) \rVert_{L_p(D)} = 0$ for Lebesgue\nobreakdash-\hspace{0pt}a.e.\ $t \in (0, T)$.
    \item
    $u$ is \emph{weakly measurable} if for any $v \in L_{p'}(D)$ the function
    \begin{equation*}
        [\, t \mapsto \langle u(t), v \rangle_{L_p(D), L_{p'}(D)} \,]
    \end{equation*}
    is $(\mathfrak{L}((0, T)), \mathfrak{B}(\RR))$\nobreakdash-\hspace{0pt}measurable.
\end{itemize}
\begin{theorem}
\label{thm:equiv-measurable}
    For $u \colon (0, T) \to L_p(D)$ measurability, strong measurability and weak measurability are equivalent.
\end{theorem}
The equivalence of strong and weak measurability is a consequence of Pettis's Measurability Theorem (see, e.g., \cite[Thm.~2.1.2]{DiUhl}).  For the fact that measurability implies strong measurability see, e.g., \cite[Thm.~1]{Vara}, whereas the proof of the reverse implication is a simple exercise.

For our purposes we will use the following definitions (see, e.g.\ \cite[Sect.~X.4]{AmEsch}).  A measurable $u \colon (0, T) \to L_p(D)$ belongs to $L_r((0, T), L_p(D))$, $1 \le r < \infty$, if $\norm{u(\cdot)}_{L_p(D)}$ belongs to $L_r((0, T))$, with
\begin{align*}
 \norm{u}_{L_r((0, T), L_p(D))}  = \biggl( \int\limits_{0}^{T} \norm{u(t)}_{L_p(D)}^r \, \dd t \biggr)^{\!\! 1/r}.
\end{align*}
Similarly, a measurable $u \colon (0, T) \to L_p(D)$ belongs to $L_{\infty}((0, T), L_p(D))$, if $\norm{u(\cdot)}_{L_p(D)}$ belongs to $L_{\infty}((0, T))$, with
\begin{align*}
 \norm{u}_{L_{\infty}((0, T), L_p(D))}  = \esssup_{t \in (0, T)}{\norm{u(t)}_{L_p(D)}}\ .
\end{align*}

The following result, a part of~{\textup{\cite[Lemma~III.11.16]{DS}}}, will be used several times.

\begin{lemma}
~
\label{lm:Dunford-Schwartz}
\begin{enumerate}[label=\textup{(\alph*)},ref=\ref{lm:Dunford-Schwartz}\textup{(\alph*)}]
    \item\label{DS_a}
    If $u \in L_1((0, T), L_1(D))$ then the function
        \begin{equation*}
            \bigl[\, (0, T) \times D \ni (t, x) \mapsto u(t)[x] \in \RR \,\bigr]
        \end{equation*}
        belongs to $L_1((0, T) \times D, \RR)$.
    \item\label{DS_b}
     If $w$ is $(\mathfrak{L}((0, T) \times D), \mathfrak{B}(\RR))$\nobreakdash-\hspace{0pt}measurable, and for Lebesgue\nobreakdash-\hspace{0pt}a.e.\ $t \in (0, T)$ the \emph{$t$\nobreakdash-\hspace{0pt}section}
        $w(t, \cdot)$ belongs to $L_p(D)$, where $1 \le p < \infty$, then the function
        \begin{equation*}
            \Bigl[\, (0, T) \ni t \mapsto \bigl[\, D \ni x \mapsto w(t, x) \in \RR \,\bigr] \, \Bigr]
        \end{equation*}
        is $(\mathfrak{L}(0, T), \mathfrak{B}(L_p(D)))$\nobreakdash-\hspace{0pt}measurable.
\end{enumerate}
\end{lemma}

\begin{remark}
\label{remark:measurable_complement}
\textup{Regarding Lemma~\ref{DS_b}, we remark  that in~\cite{DS} the analog of $w$ is assumed to be $(\mathfrak{L}((0, T)) \otimes \mathfrak{L}(D),\mathfrak{B}(\RR))$\nobreakdash-\hspace{0pt}measurable (no completion) rather than $(\mathfrak{L}((0, T) \times D),\mathfrak{B}(\RR))$\nobreakdash-\hspace{0pt}measurable.  As an $(\mathfrak{L}((0, T) \times D), \mathfrak{B}(\RR))$\nobreakdash-\hspace{0pt}measurable function can be made into an $(\mathfrak{L}((0, T)) \otimes \mathfrak{L}(D), \mathfrak{B}(\RR))$\nobreakdash-\hspace{0pt}measurable function by changing its values on a set of $(N + 1)$\nobreakdash-\hspace{0pt}dimensional Lebesgue measure zero (see, e.g., \cite[Prop.~2.12]{Fol}), our formulation follows.}
\end{remark}

\section{Assumptions and Definitions}
\label{sect:assumptions}
\subsection{Main Equation}

Consider a linear second order partial differential equation with bounded delay

\begin{equation*}
\label{main-eq}
\tag{ME}
\begin{aligned}
\frac{\partial u}{\partial t}&= \sum_{i=1}^{N}\frac{\partial}{\partial x_i }\nbbb{\sum_{j=1}^{N} a_{i j}(t, x) \frac{\partial u}{\partial x_j }+a_i(t,x)u }\\
&+\sum_{i=1}^{N} b_{i}(t, x) \frac{\partial u }{\partial x_i }+c_{0}(t, x) u\\[2ex]
&+c_{1}(t, x) u(t-R(t)); \quad 0\le  t\le  T, \ x \in D.
\end{aligned}
\end{equation*}
The delay map $R \colon [0,T]\to\RR$ is bounded from below by $0$ and from above by $1$, i.e.
\begin{equation*}
    0 \le R(t)\le 1, \qquad  \forall\, t \in [0, T].
\end{equation*}
 Sometimes the function $\xi\mapsto \xi-R(\xi)$ will be denoted by $\Phi$. The function $\Phi$ will be called \textit{relative time delay}. Further, $D \subset \RR^N$ is a bounded domain with boundary $\p D$. The equation~\eqref{main-eq} will be complemented with boundary conditions
\begin{equation}
\label{main-bc}
\tag{BC}
\mathcal{B} u = 0, \qquad  0\le t\le T, \ x \in \p D.
\end{equation}
Later on, we will use the notation $\mathcal{B}_a$ in other to exhibit dependence of the operator $\mathcal{B}$ on $a$. The boundary conditions operator~\eqref{main-bc} will be one of this form
\begin{equation*}
 \mathcal{B} u=\left\{\begin{array}{lll}
u & \quad \mathrm{(Dirichlet) } \\[2ex]
\displaystyle  \sum_{i=1}^{N}\nbbb{\sum_{j=1}^{N} a_{i j}(t, x)\frac{\partial u(t)}{\partial x_j}+a_i(t,x)u} \nu_{i} & \quad \mathrm{(Neumann)} \\
\displaystyle  \sum_{i=1}^{N}\nbbb{\sum_{j=1}^{N} a_{i j}(t, x) \frac{\partial u(t)}{\partial x_j}+a_i(t,x)u} \nu_{i}+d_{0}(t, x) u  & \quad \mathrm{ (Robin). }
\end{array}\right.
\end{equation*}

\noindent
The vector $\boldsymbol{\nu} = (\nu_{1},\dots,\nu_{N})$ denotes the unit normal on the boundary $\p D$ pointing
out of $D$, interpreted in a certain weak sense (in the regular sense if $\p D$ is sufficiently smooth~\cite{Monogr}).

The initial condition is considered in the following way: for $u_0 \in C([- 1, 0], \linebreak L_p(D))$, where $1 \le p \le \infty$, find a solution of \eqref{main-eq}$+$\eqref{main-bc} satisfying
  \begin{equation}
  \label{main-ic}
  \tag{IC}
 u(t) =  u_0(t) \text{ for } t \in [- 1, 0].
  \end{equation}
\noindent
By \eqref{main-eq}$+$\eqref{main-bc} we understand equation \eqref{main-eq} equipped with boundary condition \eqref{main-bc}. Later on, we will also use $\eqref{main-eq}_a+\eqref{main-bc}_a$ notation to indicate that parameters of $\eqref{main-eq}+\eqref{main-bc}$ are fixed to be $a$.

 Note that, without any additional assumptions on the delay map $R$, the initial data cannot be taken from $L_p(D) \oplus L_r((- 1, 0), L_p(D))$, as in~~\cite{Mi-No-Ob18} or~\cite{Mi-No-Ob20}. The reason for this is that the delay map $R$ can be constructed in such way that $t\mapsto t-R(t)$ would be a constant function. In such a situation the initial value problem~$\eqref{main-eq}+\eqref{main-bc}$ would be not meaningful. Under some additional assumptions the situation can change, for example a constant delay map allows us to introduce generalized initial data in $L_p(D) \oplus L_r((- 1, 0), L_p(D))$. However, in this paper we will not focus on that.

 In order to clearly define the problem $\eqref{main-eq}_{a}+\eqref{main-bc}_a$, it is also necessary to set the delay map $R$. However, the assumptions on $R$ will be given later. Moreover, we suppress the notation of $R$ from $\eqref{main-eq}_{a}+\eqref{main-bc}_a$.  We will present the solutions of $\eqref{main-eq}_{a}+\eqref{main-bc}_a$ in the form of $u(\cdot;a,u_0,R)$ and often suppress the notation of $a,u_0$ or $R$ if it does not lead to confusion.

\subsection{Main Assumptions}
\noindent
We introduce some assumptions on the domain $D\subset\RR^N$ and the coefficients of the problem~\eqref{main-eq}+\eqref{main-bc}.

\begin{enumerate}[label=$\mathrm{(DA\hspace{.1em}\arabic*)}$,ref=$\mathrm{DA\hspace{.1em}\arabic*}$]
\item\label{a: boundary regularity}
\medskip\noindent
 (Boundary regularity) {\em For Dirichlet boundary conditions, $D$ is a
bounded domain. For Neumann or Robin boundary conditions, $D$ is a bounded domain with Lipschitz boundary.
}
\end{enumerate}
In all expressions of the type ``a.e.'' we consider $1$\nobreakdash-\hspace{0pt}dimensional Lebesgue measure on $(0, T)$, $N$\nobreakdash-\hspace{0pt}dimensional Lebesgue measure on $D$ and $(N - 1)$\nobreakdash-\hspace{0pt}dimensional Hausdorff measure on $\partial D$.  The latter is, by~\eqref{a: boundary regularity}, equal to surface measure on $\p D$.

The notation $L_{\infty}(\p D)$ [resp.\ $L_{\infty}((0, T) \times \p D)$] corresponds to surface measure on $\p D$ [resp.\ to the product of $1$\nobreakdash-\hspace{0pt}dimensional Lebesgue measure on $(0, T)$ and surface measure on $\p D$].

\begin{enumerate}[resume,label=$\mathrm{(DA\hspace{.1em}\arabic*)}$,ref=$\mathrm{DA\hspace{.1em}\arabic*}$]

\item\label{a: boundedness} (Boundedness) {\em The functions
\begin{itemize}[label=$\diamond$]
    \item
    $a_{ij} \colon (0,T) \times D \to \RR$ \textup{(}$i, j = 1, \dots, N$\textup{)},
    \item
    $a_{i} \colon (0,T) \times D \to \RR$ \textup{(}$i = 1, \dots, N$\textup{)},
    \item
    $b_{i} \colon (0,T) \times D \to \RR$ \textup{(}$i = 1, \dots, N$\textup{)},
    \item
    $c_0 \colon (0,T) \times D \to \RR$ ,
    \item
    $c_1 \colon (0,T) \times D \to \RR$
\end{itemize}
belong to $L_{\infty}((0,T)\times D)$. When the Robin boundary condition holds the function $d_0 \colon (0,T) \times \partial D \to \RR$ belongs to $L_{\infty}((0,T)\times \partial D)$.}
\end{enumerate}
\noindent
It is worth noticing at this point that uniform $L_{\infty}(D)$-boundedness of $a_{ij}(t,\cdot)$, $a_{i}(t,\cdot)$, $b_{i}(t,\cdot)$, $c_{0}(t,\cdot)$, $c_1(t,\cdot)$ and uniform $L_{\infty}(\partial D)$-boundedness of $d_0(t,\cdot)$ for a.e.\ $t\in(0,T)$ follow from the assumption \eqref{a: boundedness} and Fubini's theorem.

\medskip
\begin{definition}[$Y$ coefficients space]
\label{def:Y}
Let $Y$ be a subset of the Banach space $L_{\infty}((0,T) \times D, \mathbb{R}^{N^2 + 2N +2})  \oplus  L_{\infty}((0,T) \times \partial D, \mathbb{R})$ satisfying the following assumptions
\begin{enumerate}[label=$\mathrm{(Y\hspace{.1em}\arabic*)}$,ref=$\mathrm{Y\hspace{.1em}\arabic*}$]
    \item\label{Y1} $Y$ is norm\nobreakdash-\hspace{0pt}bounded and, moreover, it is closed \textup{(}hence compact, via the Banach--Alaoglu theorem\textup{)} in the \WST topology,
    \item\label{Y2} the function $d_0 \ge 0$ if the Robin boundary condition holds. The function $d_0$ is interpreted as the zero function in the Dirichlet or Neumann cases.
\end{enumerate}
\end{definition}
Elements of $Y$ will be denoted by
\begin{equation*}
    a \vcentcolon=\nb{ \n{a_{i j}}_{i, j=1}^{N},\n{a_{i}}_{i=1}^{N},\n{b_{i}}_{i=1}^{N}, c_0,c_1,d_{0} }\in Y.
\end{equation*}

The \WST topology of the space $Y$ is understood in the standard sense, namely, as the \WST topology induced via the isomorphism
\begin{multline*}
      L_{\infty}((0,T) \times D, \mathbb{R}^{N^2 + 2N +2}) \oplus  L_{\infty}((0,T) \times \partial D, \mathbb{R}) \\
     \cong \big( L_{1}((0,T)) \times D, \mathbb{R}^{N^2 + 2N +2}) \oplus L_{1}((0,T) \times \partial D, \mathbb{R})\big)^*.
\end{multline*}

\begin{definition}[Flattening $Y$ to $Y_{0}$]
The mapping defined on $Y$ by
\begin{equation*}
   \nb{ \n{a_{i j}}_{i, j=1}^{N},\n{a_{i}}_{i=1}^{N},\n{b_{i}}_{i=1}^{N}, c_0,c_1,d_{0} }\, \widetilde{} \vcentcolon=  \nb{ \n{a_{i j}}_{i, j=1}^{N},\n{a_{i}}_{i=1}^{N},\n{b_{i}}_{i=1}^{N}, c_0,0,d_{0} }
\end{equation*}
will be called the \emph{flattening} of $a\in Y$.
\end{definition}
The above mapping is obviously continuous.  As a consequence, the image $Y_0$ of $Y$ under that mapping shares properties analogous to \eqref{Y1} and \eqref{Y2}.

\begin{enumerate}[resume,label=$\mathrm{(DA\hspace{.1em}\arabic*)}$,ref=$\mathrm{DA\hspace{.1em}\arabic*}$]
\item\label{Ellipticity} (Ellipticity)
{\em There exists a constant $\alpha_0 > 0$ such that for any $a_0 \in Y_0$ the inequality
\begin{equation*}
    \displaystyle \sum\limits_{i,j=1}^{N} a_{ij}(t, x) \xi_{i} \xi_{j} \ge \alpha_0 \sum\limits_{i=1}^{N} \xi_{i}^2,
\end{equation*}
holds for a.e.\ $(t,x) \in (0,T) \times D$ and all $\xi \in \RR^{N}$, and the functions $a_{ij}(\cdot,\cdot)$ are symmetric in the indices, i.e. $a_{ij}(\cdot,\cdot) \equiv a_{ji}(\cdot,\cdot)$ for all $i, j = 1, \dots, N$.}

\item\label{a: compactness of Y } (Sequential compactness of $Y_0$ with respect to convergence a.e.)

{\em Any sequence $(a_{0,m})_{m = 1}^{\infty}$ of elements of $Y_0$, where
\begin{equation*}
  a_{0, m} \vcentcolon= \bigl( (a_{ij,m})_{i,j=1}^{N}, (a_{i,m})_{i=1}^{N}, (b_{i,m})_{i=1}^{N}, c_{0,m}, 0, d_{0,m} \bigr),
\end{equation*}
convergent as $m \to \infty$ in the \WST topology to $a_0 \in Y_0$ has the property that}
\begin{itemize}[label=$\bullet$]
\item
{\em the sequence
$\bigl( (a_{ij,m})_{i,j=1}^{N}, (a_{i,m})_{i=1}^{N}, (b_{i,m})_{i=1}^{N} \bigr)$ converges to  $ \bigl( (a_{ij})_{i,j=1}^{N}, \linebreak (a_{i})_{i=1}^{N}, (b_{i})_{i=1}^{N} \bigr)$ pointwise a.e.\ on $(0,T) \times D$,}
\item
{\em the sequence
$d_{0,m}$ converges to $d_{0}$ pointwise a.e.\ on $(0,T) \times \partial D$.
}
\end{itemize}

\medskip\noindent

\textup{Occasionally we will use the following.}
\item\label{a:only_c_continuous}
{\em $Y_0$ is a singleton.}
\end{enumerate}
For the purposes of studying continuous dependence on parameters and delay we introduce now the delay class and the relative delay class equipped with suitable topologies.
\begin{definition}
The delay class is defined as follows
 \begin{equation*}
  \mathcal{R}\vcentcolon=\big\{R\in L_{\infty}((0,T)): \, R(t)\in [0,1] \text{ for a.e. } t\in (0,T) \big\}.
 \end{equation*}
\noindent
The delay class is equipped with the \WST topology.
\end{definition}
\noindent
At some moments we also use the following notation $\widetilde{\mathcal{R}} \vcentcolon= \{[\,t\mapsto t-R(t)\,]: \, R\in \mathcal{R} \}$ especially in more abstract lemmas when general properties of the mapping $t\mapsto t-R(t)$ are important.

\begin{remark}
\label{R_class_compactnes}
 The delay class $\mathcal{R}$ is a norm\nobreakdash-\hspace{0pt}bounded, convex and weak\nobreakdash-\hspace{0pt}\textup{*} closed subset of $ L_{\infty}((0,T))$, hence, by the Banach--Alaoglu theorem, it is compact in the \WST topology.
\end{remark}

The following assumption is a property of a subset $\mathcal{R}_0\subset\mathcal{R}$.
\begin{enumerate}[resume,label=$\mathrm{(DA\hspace{.1em}\arabic*)}$,ref=$\mathrm{DA\hspace{.1em}\arabic*}$]
\item\label{a:tau_pointwise_convergence}
{\em If $R\in \mathcal{R}_0$ is a \WST limit of $(R_m)_{m=1}^{\infty}\subset \mathcal{R}_0$ then $(R_m)_{m=1}^{\infty}$ converge pointwise a.e.\ on $(0,T)$ to $R$.}
\end{enumerate}

\begin{remark}
Note that the assumption~\eqref{a:tau_pointwise_convergence} is naturally satisfied when $\mathcal{R}_0$ is compact in $\mathcal{R}$ with respect to the norm topology. This fact follows from an observation that any Hausdorff topology weaker than the norm topology \textup{(}such as the \WST topology\textup{)} is equal to the norm topology on a compact subset.
\end{remark}

\begin{remark}
Note that the \WST topology on $Y$ is metrizable, see~\textup{\cite[(1.3.1)]{Monogr}}. Similarly, the \WST topology on $\mathcal{R}$ is metrizable, see \textup{\cite[Thm. 3.6.17 and Cor. 3.6.18]{DM}}.
\end{remark}

\section{Weak Solutions}
\label{sect:weak-sol}
In the present section we assume \eqref{a: boundary regularity}, \eqref{a: boundedness} and that the flattening $Y_0$ of $Y$ as in Definition~\ref{def:Y} satisfies \eqref{Ellipticity}.  Occasionally we will assume \eqref{a: compactness of Y }.

We start with a PDE parameterized by $a_0 \in Y_{0}$
\begin{equation*}
\label{main-eq-pert-delay}
\tag{$\widehat{\mathrm{ME}}$}
    \begin{aligned}
        \frac{\partial u}{\partial t}&= \sum_{i=1}^{N}\frac{\partial}{\partial x_i }\nbbb{\sum_{j=1}^{N} a_{i j}(t,x) \frac{\partial u}{\partial x_j }+a_i(t,x)u }\\
&+\sum_{i=1}^{N} b_{i}(t, x) \frac{\partial u}{\partial x_i }+c_{0}(t, x ) u; \quad 0 \le s\le t\le T, \ x\in D.
    \end{aligned}
\end{equation*}
Equation~\eqref{main-eq-pert-delay} is complemented with boundary conditions
\begin{equation}
\label{bc-higher_orders}
\tag{$\widehat{\mathrm{BC}}$}
\mathcal{B} u = 0, \quad 0 \le s \le t \le T, \ x \in \p D,
\end{equation}
where $\mathcal{B}$ is either the Dirichlet or the Neumann or else the Robin boundary operator.

We are looking for solutions of the problem \textup{(\ref{main-eq-pert-delay})$_{a_0}$+
(\ref{bc-higher_orders})$_{a_0}$} for initial condition $u_0\in L_2(D)$.  To define a solution we introduce the space $H$ as follows. Let
\begin{equation*}
V \vcentcolon=\left\{\begin{array}{lll}
H^{1}_{0}(D) & \quad \mathrm{for \ Dirichlet \ boundary \ condition } \\[1ex]
H^{1}(D) & \quad \mathrm{for \ Neumann \ or \  Robin\ boundary \ condition}
\end{array}\right.
\end{equation*}
and
\begin{equation*}
\label{W-space-eq}
H = H(s,T;V) \vcentcolon= \{\,v \in L_2((s,T),V): \dot{v} \in L_2((s,T),V^{*})\,\}
\end{equation*}
equipped with the norm
\begin{equation*}
\lVert v \rVert_{H} \vcentcolon= \Bigl( \int_s^{T} \lVert v(\zeta) \rVert_{V}^2 \, \dd\zeta  + \int_s^{T} \lVert \dot{v}(\zeta) \rVert_{V^{*}}^2 \, \dd\zeta  \Bigr)^{\frac{1}{2}},
\end{equation*}
where $\dot{v} \vcentcolon= \dd v/\dd t$ is the time derivative in the sense of distributions taking values in $V^{*}$  (see \cite[Chpt.~XVIII]{DL} for definitions).

\medskip

For $a_0 \in Y_{0}$ define a bilinear form
\begin{equation*}
\begin{aligned}
B_{a_0}[t; u, v] & \vcentcolon= \int\limits_{D} \biggl( \, \sum\limits_{i, j = 1}^{N}a_{ij}(t, x) \frac{\partial u}{\partial x_{i}} \frac{\partial v}{\partial x_{j}} + \sum\limits_{i = 1}^{N} a_{i}(t, x) u \frac{\partial v}{\partial x_{i}}
\\
& - \sum\limits_{i = 1}^{N} b_{i}(t, x) \frac{\partial u}{\partial x_{i}} v - c_0(t,x)uv\biggr) \, \dd x
\end{aligned}
\end{equation*}
in the Dirichlet or Neumann case, and
\begin{equation*}
\begin{aligned}
B_{a_0}[t; u, v] & \vcentcolon= \int\limits_{D} \biggl( \, \sum\limits_{i, j = 1}^{N} a_{ij}(t, x) \frac{\partial u}{\partial x_{i}} \frac{\partial v}{\partial x_{j}} + \sum\limits_{i = 1}^{N} a_{i}(t, x) u\frac{\partial v}{\partial x_{i}}
\\
& - \sum\limits_{i = 1}^{N} b_{i}(t, x) \frac{\partial u }{\partial x_{i}} v-c_0(t,x)uv\biggr) \, \dd x + \int\limits_{\partial D} d_0(t, x) u v \, \dd H
\end{aligned}
\end{equation*}
in the Robin case, where $H_{N - 1}$ stands for the $(N-1)$-dimensional Hausdorff measure.

\begin{definition}[Local Weak Solution]
\label{loc-weak-solution-def-null}
For $a_0 \in Y_{0}$, $0 \le s\le t\le T$ and $u_0 \in L_2(D)$ a function $u \in L_2([ s, t ],V)$ such that $\dot{u} \in L_2( [ s, t ], V^{*} )$ is a {\em weak solution\/} of \textup{(\ref{main-eq-pert-delay})$_{a_0}$+
(\ref{bc-higher_orders})$_{a_0}$} {\em on $[s, t]$ with initial condition $u(s) = u_0$} if
\begin{equation*}
  - \int\limits_{s}^{t} (u(\zeta), v)_{L_2(D)} \, \dot{\psi}(\zeta) \, \dd\zeta  + \int\limits_{s}^{t} B_{a_0}[\zeta; u(\zeta), v]  \psi(\zeta) \, \dd\zeta  = (u_0, v)_{L_2(D)} \, \psi(s)
\end{equation*}
for any $v \in V$ and any $\psi\in \mathcal{D}([s, t),\RR)$ where set $\mathcal{D}([s, t), \RR)$ is the space of all smooth real functions having compact support in $[s,t)$ and $(\cdot, \cdot)_{L_2(D)}$ denotes the standard inner product in $L_2(D)$.
\end{definition}

\begin{definition}[Global Weak Solution]
When $t=T$ in definition~\textup{\ref{loc-weak-solution-def-null}} then a weak solution will be called a global weak solution.
\end{definition}

\begin{proposition} [Existence of global weak solution]
For any initial condition $u_0\in L_2(D)$ there
exists a unique global weak solution of~\eqref{main-eq-pert-delay}$+$\eqref{bc-higher_orders}.
\end{proposition}
\begin{proof}
See \cite[Thm. 2.4]{Dan2} for a proof and \cite[Prop. 2.1.5]{Monogr} for a unified theory of weak solutions. \qed
\end{proof}

For $a_0\in Y_0$ and $0\le s<T$ we write the unique global weak solution of \textup{(\ref{main-eq-pert-delay})$_{a_0}$+
(\ref{bc-higher_orders})$_{a_0}$} with initial condition $u(s)=u_0$ as $\U_{a_0} (t,s)u_0\vcentcolon=u(t)$.

Below we present a couple of results from~\cite[Ch.~2]{Monogr}.
\begin{proposition}
  \label{prop:higher_orders-skew_product}
The mappings
  \begin{equation*}
    U_{a_0}(t,s) u_0 = u(t; a_0, u_0), \quad 0 \le s\le t\le T, \ a_0 \in Y_0, \ u_0 \in L_2(D)
  \end{equation*}
  have the following properties.
  \begin{gather}
  \label{eq:cocycle2-1}
  U_{a_0}(s, s) = \mathrm{Id}_{L_2(D)}, \quad a_0 \in Y_0, \ s \in [0,T],
  \\
  \label{eq:cocycle2-2}
  U_{a_0}(t_2, t_1) \circ U_{a_0}(t_1, s) = U_{a_0}(t_2, s), \quad a_0 \in Y_0, \  0 \le s \le t_1 \le t_2\le T.
  \end{gather}
\end{proposition}
\begin{proof}
 See~\cite[Props.~2.1.5 through 2.1.8]{Monogr}. \qed
\end{proof}

\begin{proposition}
~
  \label{prop:higher_orders-skew_product-p}
   \begin{enumerate}[label=\textup{(\roman*)}, ref=\textup{\ref{prop:higher_orders-skew_product-p}}\textup{(\roman*)}]
     \item\label{prop:higher_orders-skew_product-p_i}
      Let $1 \le p < \infty$ and $0 \le s < T$.  For any $a_0 \in Y_0$ there exists $U_{a_{0},p}(t) \in \mathcal{L}(L_p(D))$ such that
      \begin{equation*}
          U_{a_{0},p}(t,s) u_0 = U_{a_{0}}(t,s) u_0, \quad u_0 \in L_2(D) \cap L_p(D).
      \end{equation*}
      \item\label{prop:higher_orders-skew_product-p_ii}
      Let $1 < p < \infty$ and $a_{0} \in Y_0$.  Then the mapping
      \begin{equation*}
          \bigl[\, [s, T] \ni t \mapsto U_{a_{0},p}(t,s) \in \mathcal{L}_{\mathrm{s}}(L_p(D)) \,\bigr]
      \end{equation*}
      is continuous.
 \end{enumerate}
\end{proposition}
\begin{proof}
  See~\cite[Cor.~7.2]{Dan2} for part~(i) and~\cite[Thm.~5.1]{Dan2} for part~(ii). \qed
\end{proof}

For $p=1$ we have an analog of Proposition~\ref{prop:higher_orders-skew_product-p_ii}.
\begin{proposition}
\label{prop:higher_orders-p-q-continuity}
 Let $1 \le p < \infty$, $0 \le s < T$ and $a_0 \in Y_0$. Then the mapping
    \begin{equation*}
        \big[\, (s, T] \ni t \mapsto U_{a_0}(t, s) \in \mathcal{L}_{\mathrm{s}}(L_p(D)) \, \big]
    \end{equation*}
    is continuous.
\end{proposition}
\begin{proof}
    See~\cite[Prop.~2.2.6]{Monogr}. \qed
\end{proof}

For $0 \le s < T$ we write $U_{a_0,p}(s,s) = \mathrm{Id}_{L_p(D)}$ even if $p = 1, \infty$.

\begin{proposition}
\label{Extension_identyty_cocycle}
For any $a_0\in Y_{0}$, $0 \le s\le t_1\le t_2\le T$ and any $1\le p\le \infty$
\begin{equation}
\label{eq:cocycle2-2_p}
    \U_{a_0,p}(t_2,t_1)\circ \U_{a_0,p} (t_1,s)= \U_{a_0,p}(t_2,s)
\end{equation}
\end{proposition}

\begin{proof}
See ~\cite[Prop.~2.1.7]{Monogr} for the proof of $p=2$ case. For $p\not=2$ it suffices to use the fact that $U_{a_0}(t,s)\in\mathcal{L}(L_2(D))$ and the continuity of the mappings $[u\mapsto \U_{a_0,p}(t_2,t_1)\circ \U_{a_0,p} (t_1,s)u]$ and $[u\mapsto \U_{a_0,p}(t_2,s)u]$, which is guaranteed by Proposition~\ref{prop:higher_orders-skew_product-p}. \qed
 \end{proof}

\begin{proposition}
\label{U_Kernel}
For any $a_0\in Y_{0}$ and any $0 \le s\le t_1\le t_2\le T$ the operator $ \U_{a_0}(t_2,t_1)$ has an a.e.\ nonnegative kernel.
\end{proposition}

\begin{proof}
    See~\cite[Thm. 1.3]{ArBu} for the existence of a kernel, for nonnegativity see~\cite[Cor. 8.2]{Dan2}. \qed
\end{proof}

\begin{proposition}
  \label{prop:higher_orders-skew_product-p-q-estimates}~
  \begin{enumerate}[label=\textup{(\roman*\textup)},ref=\ref{prop:higher_orders-skew_product-p-q-estimates}\textup{(\roman*)}]
    \item\label{prop:higher_orders-skew_product-p-q-estimates_1} For any $a_{0} \in Y_0$, any $(s,t) \in \dot{\Delta}$ and any $1 \le p \le q \le \infty$ there holds $U_{a_{0}}(t,0) \in \mathcal{L}(L_p(D), L_q(D))$.
    \item\label{prop:higher_orders-skew_product-p-q-estimates_2} There are constants $M \ge 1$ and $\gamma \in \RR$ such that
    \begin{equation}
    \label{eq:L_p-L_q estymation}
           \left\|U_{a_0}(t, s)\right\|_{\mathcal{L}(L_{p}(D), L_{q}(D))} \leq M (t-s)^{-\frac{N}{2}\left(\frac{1}{p}-\frac{1}{q}\right)} e^{\gamma (t-s)}
    \end{equation}
    for $1 \le p \le q \le \infty$, $a_{0} \in Y_0$ and $(s,t)\in\dot{\Delta}$.
\end{enumerate}
\end{proposition}
\begin{proof}
  See~\cite[Sect.~5 and Cor.~7.2]{Dan2}. \qed
\end{proof}

In~particular, setting $p = q$ we have
\begin{equation}
\label{eq:L_p estymation+}
     \left\|U_{a_0}(t, s)\right\|_{\mathcal{L}(L_{p}(D))} \leq M  e^{\gamma (t-s)}.
\end{equation}
 In the sequel we will frequently assume that $\gamma \ge 0$ in Proposition~\ref{prop:higher_orders-skew_product-p-q-estimates} and its derivates.

\begin{proposition}
\label{prop:local-regularity-higher_order}
Let $1 \le p \le \infty$ and $0 \le s < T$. Then for any $T_1 \in (s, T]$ there exists $\alpha \in (0,1)$ such that for any $a_0 \in Y_0$, any $u_0 \in L_p(D)$, and any compact subset $D_0 \subset D$ the function
$\bigl[\,[T_1,T] \times D_0 \ni (t,x) \mapsto (U_{a_0}(t) u_0)[x]\,\bigr]$
belongs to $C^{\alpha/2,\alpha}([T_1,T] \times D_0)$. Moreover,
for fixed $T_1$, and $D_0$, the $C^{\alpha/2,\alpha}([T_1,T]
\times D_0)$\nobreakdash-norm of the above restriction is bounded
above by a constant depending on $\norm{u_0}_{L_p(D)}$ only.
\end{proposition}
\begin{proof}
It follows from Proposition~\ref{prop:higher_orders-skew_product-p-q-estimates} and from
\cite[Chpt.~III, Thm.~10.1]{LaSoUr}. \qed
\end{proof}

\begin{proposition}
  \label{prop:higher_order-compactness}
 For any $(s,T_1)\in \dot{\Delta}$, $1 \le p < \infty$ and a bounded $E \subset L_p(D)$ the set
  \begin{equation*}
    \bigl\{\, \bigl[\, [T_1, T] \ni t \mapsto U_{a_0}(t,s) u_0 \in L_p(D) \bigr] \vcentcolon a_0 \in Y_0, u_0 \in E \,\bigr\}
  \end{equation*}
  is precompact in $C([T_1, T], L_p(D))$.
\end{proposition}
\begin{proof}
  Fix $(s,T_1)\in \dot{\Delta}$, $1 \le p < \infty$ and a bounded $E \subset L_p(D)$. Let $(a_{0,m})_{m=1}^{\infty} \subset Y_0$ and $(u_{0,m})_{m=1}^{\infty} \subset E$.  Put, for $m = 1, 2, \dots$,
  \begin{equation*}
      u_m(t) \vcentcolon =U_{a_{0,m}}(t) u_{0,m}, \quad t \in [T_1, T].
  \end{equation*}
  It follows from Proposition~\ref{prop:local-regularity-higher_order} via the Ascoli--Arzel\`a theorem by diagonal process that, after possibly taking a subsequence, $(u_m)_{m = 1}^{\infty}$ converges as $m \to \infty$ to some function $\tilde{u}$ defined on $[T_1, T]$ and taking values in the set of continuous real functions on $D$ in such a way that for any compact $D_0 \subset D$ the functions $[\, t \mapsto u_m(t)\!\!\restriction_{D_0} \,]$ converge to $[\, t \mapsto \tilde{u} (t)\!\!\restriction_{D_0} \,]$ in $C([T_1, T], C(D_0))$.

  We claim that $u_m$ converge to $\tilde{u}$ in the $C([T_1, T], L_p(D))$\nobreakdash-\hspace{0pt}norm. By Proposition~\ref{prop:higher_orders-skew_product-p-q-estimates}, there is $M > 0$ such that $\norm{u_m(t)}_{L_{\infty}(D)} \le M$ and $\norm{\tilde{u}(t)}_{L_{\infty}(D)} \le M$ for all $m = 1, 2, \ldots$ and all $t \in [T_1, T]$.  For $\epsilon > 0$ take a compact $D_0 \subset D$ such that $\lambda(D \setminus D_0) < (\epsilon/(4 M))^p$, where $\lambda$ denotes the $N$\nobreakdash-\hspace{0pt}dimensional Lebesgue measure. We have
  \begin{equation*}
      \norm{(u_m(t) - \tilde{u}(t)) \, \mathbbm{1}_{\! D \setminus D_0}}_{L_p(D)} \le \frac{\epsilon}{2}
  \end{equation*}
  for all $m = 1, 2, \ldots$ and all $t \in [T_1, T]$.  Further, since $[\, t \mapsto u_m(t)\!\!\restriction_{D_0} \,]$ converge to $ [\, t \mapsto \tilde{u}(t)\!\!\restriction_{D_0} \,]$ in the $C([T_1, T], C(D_0))$\nobreakdash-\hspace{0pt}norm, there is $m_0$ such that
  \begin{equation*}
      \norm{(u_m - \tilde{u}) \mathbbm{1}_{\! D_0}}_{C([T_1, T], L_p(D))} \le \frac{\epsilon}{2}
  \end{equation*}
  for all $m \ge m_0$ (here $\mathbbm{1}_{\! D_0}$ stands for the function constantly equal to $\mathbbm{1}_{\! D_0}$).  Consequently,
  \begin{equation*}
      \norm{u_m - \tilde{u}}_{C([T_1, T], L_p(D))} \le \epsilon
  \end{equation*}
  for all $m \ge m_0$. \qed
 \end{proof}

\begin{corollary}
\label{cor:p-q}
    Let $1 \le p < \infty$, $0 \le s < T$, $a_0 \in Y_0$.  Then the mapping
    \begin{equation*}
        [\, (s,T] \times L_p(D) \ni (t, u_0) \mapsto U_{a_0}(t, s) u_0 \in L_p(D) \,]
    \end{equation*}
    is continuous.
\end{corollary}
\begin{proof}
   Let $(u_m)_{m = 1}^{\infty}$ converge in $L_p(D)$ to $u_0$ and let $(t_m)_{m = 1}^{\infty}$ converge to $t > s$.  Take $\epsilon > 0$.  It follows from Proposition~\ref{prop:higher_orders-p-q-continuity} that there is $m_1$ such that $\norm{U_{a_0}(t_m, s) u_0 - U_{a_0}(t, s) u_0}_{L_p(D)} < \epsilon/2$ for $m \ge m_1$, and it follows from Proposition~\ref{prop:higher_orders-skew_product-p-q-estimates_2} that there is $m_2$ such that $\norm{U_{a_0}(t_m, s) u_m - U_{a_0}(t_m, s) u_0}_{L_p(D)} < \epsilon/2$ for $m \ge m_2$.  Consequently,
   \begin{multline*}
       \norm{U_{a_0}(t_m, s) u_m -U_{a_0}(t, s) u_0}_{L_p(D)} \\
       \le \norm{U_{a_0}(t_m, s) u_m - U_{a_0}(t_m, s) u_0}_{L_p(D)}
       \\+ \norm{U_{a_0}(t_m, s) u_0 - U_{a_0}(t, s) u_0}_{L_p(D)} < \epsilon
   \end{multline*}
    for $m \ge  \max\{m_1, m_2\}$. \qed
\end{proof}

\subsection{The Adjoint Operator}
For a fixed $0<s\le T$ together with \textup{(\ref{main-eq-pert-delay})$_{a_0}$+
(\ref{bc-higher_orders})$_{a_0}$} we consider the adjoint equations, that is the backward parabolic equations
\begin{align}
\label{eq-adjoint-hull}
-\frac{\partial u}{\partial t} = &\sum_{i=1}^{N}
\frac{\p }{\p x_i}\Bigl (\sum_{j=1}^{N} a_{ji}(t,x) \frac{\p u}{\p
x_j} - b_i(t,x) u \Bigr ) \nonumber \\
& - \sum_{i=1}^{N} a_i(t,x) \frac{\p u}{\p x_i}+c_0(t,x)u, \quad  0\le t<s,\ x \in D,
\end{align}
complemented with the boundary conditions:
\begin{equation}
\label{bc-adjoint-hull}
\mathcal{B}^{*}_{a_0}u = 0, \quad  0\le t<s,\
x \in \p D,
\end{equation}
where $\mathcal{B}^{*}_{a_0}u = \mathcal{B}_{a_0^*}u$ with $a_0^* \vcentcolon=
((a_{ji})_{i,j=1}^N,-(b_i)_{i=1}^N,-(a_i)_{i=1}^N, c_0,d_0)$ and
$\mathcal{B}_{a_0^*}$ is as in \eqref{bc-higher_orders}$_{a_0}$ with
$a_0$ replaced by $a_0^*$.

Since all analogs of the assumptions \eqref{a: boundary regularity} and \eqref{a: boundedness} are satisfied for~\eqref{eq-adjoint-hull}+\eqref{bc-adjoint-hull}, we can define, for $u_0 \in L_2(D)$, a global (weak) solution of~\eqref{eq-adjoint-hull}$_{a^*_0}$+\eqref{bc-adjoint-hull}$_{a^*_0}$, defined on $[0,s]$, with the \textit{final condition} $u(s) = u_0$. The following analog of Proposition~\ref{prop:higher_orders-skew_product} holds.
\begin{proposition}
    \label{prop:higher_orders-skew_product-adjoint}
  For $a_0 \in Y_0$, $0<s\le T$ and $u_0 \in L_2(D)$ there is precisely one global weak solution
  \begin{equation*}
   \bigl[\, [0,s] \ni t \mapsto  U^{*}_{a_0}(t,s) u_0 \in L_2(D) \,\bigr]
  \end{equation*}
  of \textup{\eqref{eq-adjoint-hull}$_{a^*_0}$+%
  \eqref{bc-adjoint-hull}}$_{a^*_0}$ satisfying the final condition $u^{*}(s; a_0, u_0) = u_0$. This  mapping has the following properties
\begin{align}
  \label{eq:cocycle2-1-adjoint-semiprocess}
  U^{*}_{a_{0}}(t, t) = \mathrm{Id}_{L_2(D)}, & \quad a_{0} \in Y_0, \ t \in [0,s],
  \\
  \label{eq:cocycle2-2-adjoint-semiprocess}
  U^{*}_{a_{0}}(t_1, t_2) \circ U^{*}_{a_{0}}(t_2, s) = U^{*}_{a_{0}}(t_1, s), & \quad a_{0} \in Y_0, \ 0 \le t_1 \le t_2 \le s.
  \end{align}
\end{proposition}

From now on $s$ and $t$ will play a role as in the \textup{(\ref{main-eq-pert-delay})$_{a_0}$+
(\ref{bc-higher_orders})$_{a_0}$}.

Below we formulate an analog of Proposition~\ref{prop:higher_orders-skew_product-p}.
\begin{proposition}
~
\label{prop:higher_orders-skew_product-p-adjoint}
\begin{enumerate}[label=\textup{(\roman*)},ref=\ref{prop:higher_orders-skew_product-p-adjoint}\textup{(\roman*)}]
      \item\label{prop:higher_orders-skew_product-p-adjoint_(i)}
      Let $1 \le p < \infty$ and $(s,t)\in\dot{\Delta}$. Then $U^{*}_{a_0}(s,t)$ extends to a linear operator in $\mathcal{L}(L_{p}(D)).$
      \item\label{prop:higher_orders-skew_product-p-adjoint_(ii)}
      Let $1 < p < \infty$, $0<s\le T$ and $a_{0} \in Y_0$.  Then the mapping
      \begin{equation*}
           \bigl[\, [0,s] \ni t \mapsto U^{*}_{a_{0}}(s,t) \in \mathcal{L}_{\mathrm{s}}(L_{p}(D)) \,\bigr]
      \end{equation*}
      is continuous.
  \end{enumerate}
\end{proposition}

The following analog of Proposition~\ref{prop:higher_orders-skew_product-p-q-estimates_1} holds.
\begin{proposition}
\label{prop:higher_orders-skew_product-p-q-adjoint}
      For any $a_{0} \in Y_0$, any $0\le t<s\le T$ and any $1 \le p \le q \le \infty$ there holds $U^{*}_{a_{0}}(s,t) \in \mathcal{L}_{\mathrm{s}}(L_p(D), L_q(D))$.
\end{proposition}

\begin{proposition}
\label{prop:dual-L2}
    For  $a_0 \in Y_0$ there holds
    \begin{multline}
    \label{eq:dual-L2-semiprocess}
        \langle U_{a_0}(t,s) u_0, v_0 \rangle_{L_2(D)} = \langle u_0, U^{*}_{a_0}(s,t) v_0 \rangle_{L_2(D)}
        \\
        \text{for any } 0\le s \le t\le T, \ u_0, v_0 \in L_{2}(D).
    \end{multline}
\end{proposition}
Proposition~\ref{prop:dual-L2} states that the linear operator $U^{*}_{a_0}(s, t) \in \mathcal{L}(L_2(D))$ is the dual (in the functional-analytic sense) of $U_{a_0}(t,s) \in \mathcal{L}(L_2(D))$.  For a proof, see~\cite[Prop.~2.3.3]{Monogr}.

\begin{proposition}
\label{prop:dual-Lp-q}
    For $1 < p < \infty$ and $a_0 \in Y_0$ there holds
\begin{equation}
    \label{eq:dual-Lp-q-semiprocess}
        \langle U_{a_0}(t,s) u_0, v_0 \rangle_{L_p(D), L_{p'}(D)} = \langle u_0, U^{*}_{a_0}(s,t) v_0 \rangle_{L_p(D), L_{p'}(D)}
\end{equation}
for any $(s,t)\in \dot{\Delta}$, $u_0 \in L_p(D)$ and $v_0 \in L_{p'}(D)$.

\end{proposition}
\begin{proof}
     Fix $(s,t)\in \dot{\Delta}$, $u_0 \in L_p(D)$ and $v_0 \in L_{q'}(D)$.  From Propositions~\ref{prop:higher_orders-skew_product-p-q-estimates_1} and \ref{prop:higher_orders-skew_product-p-q-adjoint} it follows that $U_{a_0}(\zeta,s) u_0, U^{*}_{a_0}(\zeta,t) v_0 \in L_2(D)$ for all $\zeta \in (s, t)$, consequently $\langle U_{a_0}(\zeta,s) u_0, U_{a_0}(\zeta,t) v_0 \rangle_{L_2(D)}$ is well defined for such $\zeta$.
    An application of \eqref{eq:cocycle2-2}, Proposition~\ref{prop:dual-L2} and \eqref{eq:cocycle2-2-adjoint-semiprocess} gives that for any $s < \zeta_1 \le \zeta_2 < t$ there holds
    \begin{multline*}
        \langle U_{a_0}(\zeta_2,s) u_0, U^{*}_{a_0}(\zeta_2,t) v_0 \rangle_{L_2(D)}
        \\
        = \langle U_{a_0}(\zeta_2,\zeta_1) U_{a_0}(\zeta_1,s) u_0, U^{*}_{a_0}(\zeta_2,t) v_0 \rangle_{L_2(D)}
        \\
        = \langle  U_{a_0}(\zeta_1,s) u_0, U^{*}_{a_0}(\zeta_1, \zeta_2) U^{*}_{a_0}(\zeta_2,t) v_0 \rangle_{L_2(D)}
        \\
        = \langle  U_{a_0}(\zeta_1,s) u_0, U^{*}_{a_0}(\zeta_1,t) v_0 \rangle_{L_2(D)}.
    \end{multline*}
    Therefore the assignment
    \begin{align*}
        (s, t) \ni \zeta \mapsto \langle & U_{a_0}(\zeta,s) u_0, U^{*}_{a_0}(\zeta,t) v_0 \rangle_{L_2(D)}
        \\
        = {} & \langle U_{a_0}(\zeta,s) u_0, U^{*}_{a_0}(\zeta,t) v_0 \rangle_{L_p(D), L_{p'}(D)}
        \\
        = {} & \langle U_{a_0}(\zeta,s) u_0, U^{*}_{a_0}(\zeta,t) v_0 \rangle_{L_p(D), L_{p'}(D)}
    \end{align*}
    is constant (denote its value by $A$).  If we let $\zeta \nearrow t$, then $U_{a_0}(\zeta,s) u_0$ converges, by Proposition~\ref{prop:higher_orders-p-q-continuity}, in the $L_p(D)$\nobreakdash-\hspace{0pt}norm to $U_{a_0}(t,s) u_0$ and $U^{*}_{a_0}(\zeta,t) v_0$ converges, by Proposition~\ref{prop:higher_orders-skew_product-p-adjoint_(ii)}, in the $L_{p'}(D)$\nobreakdash-\hspace{0pt}norm to $v_0$, consequently $\langle U_{a_0}(t,s) u_0, v_0 \rangle_{L_p(D), L_{p'}(D)} = A$.  If we let $\zeta \searrow s$, then $U_{a_0}(\zeta,s) u_0$ converges, by Proposition~\ref{prop:higher_orders-skew_product-p_ii}, in the $L_p(D)$\nobreakdash-\hspace{0pt}norm to $u_0$ and $U^{*}_{a_0}(\zeta,t) v_0$ converges, by Propositions~\ref{prop:higher_orders-skew_product-p-q-adjoint} and \ref{prop:higher_orders-skew_product-p-adjoint_(ii)}, in the $L_{p'}(D)$\nobreakdash-\hspace{0pt}norm to $U^{*}_{a_0}(s,t) v_0$, consequently $\langle u_0, U^{*}_{a_0}(s,t) v_0 \rangle_{L_p(D), L_{p'}(D)} = A$.  This concludes the proof. \qed
\end{proof}

It follows from Proposition~\ref{prop:dual-Lp-q} that the linear operator $\allowbreak U^{*}_{a_0}(s, t) \in \mathcal{L}(L_{p'}(D))$ is the dual (in the functional-analytic sense) of $U_{a_0}(t,s) \in \mathcal{L}(L_p(D))$.

In the light of the above, the following counterpart to Proposition~\ref{prop:higher_orders-skew_product-p-q-estimates_2} holds.

\begin{proposition}
\label{prop:higher_orders-skew_product-p-q-estimates-adjoint}
  There are constants $M \ge 1$ and $\gamma \in \RR$, the same as in Proposition~\textup{\ref{prop:higher_orders-skew_product-p-q-estimates}}, such that
    \begin{equation*}
      \lVert U^{*}_{a_{0}}(s,t) \rVert_{\mathcal{L}(L_p(D), L_q(D))} \le M (t-s)^{-\tfrac{N}{2} \bigl(\tfrac{1}{p} - \tfrac{1}{q}\bigr)} e^{\gamma (t-s)}
    \end{equation*}
    for $1 \le p \le q \le \infty$, $a_{0} \in Y_{0}$ and $(s,t)\in\dot{\Delta}$.
\end{proposition}

\subsection{Continuous Dependence of Weak Solutions}
\label{subsect:dependence-weak-solutions}
\begin{lemma}
\label{lm:continuity-in-strong-top}
    Let $1 < p < \infty$ and $a_0 \in Y_0$.  Then the mapping
    \begin{equation*}
        \big[\, \dot{\Delta} \ni (s, t) \mapsto U_{a_{0}}(t, s) \in \mathcal{L}_{\mathrm{s}}(L_p(D)) \,\big]
    \end{equation*}
    is continuous.
\end{lemma}
\begin{proof}
   If $t_m \to t > s$ as $m \to \infty$, $U_{a_0}(t_m, s)u_0 \to U_{a_{0}}(t, s)u_0$ in $L_p(D)$, by Proposition~\ref{prop:higher_orders-p-q-continuity}.

    Assume that $s_m \to s < t$ as $m \to \infty$.  Fix $u_0 \in L_p(D)$ and $v_0 \in L_{p'}(D)$.  We have, by Proposition~\ref{prop:dual-Lp-q} and the adjoint equation analog of~Proposition~\ref{prop:higher_orders-p-q-continuity},
    \begin{multline*}
        \langle U_{a_{0}}(t, s_m)u_0, v_0 \rangle_{L_p(D), L_{p'}(D)} = \langle u_0, U_{a_{0}}^{*}(s_m, t)v_0 \rangle_{L_p(D), L_{p'}(D)}
        \\
        \to \langle u_0,U_{a_{0}}^{*}(s, t)v_0 \rangle_{L_p(D), L_{p'}(D)} = \langle U_{a_{0}}(t, s)u_0, v_0 \rangle_{L_p(D), L_{p'}(D)},
    \end{multline*}
    so $U_{a_{0}}(t, s_m)u_0 \rightharpoonup U_{a_{0}}(t, s)u_0$ in $L_p(D)$.  As $\{\, U_{a_{0}}(t, s_m)u_0 : m \in \NN \,\}$ is, by Proposition~\ref{prop:higher_order-compactness}, precompact in $L_p(D)$, the convergence is in the norm.

    Finally, assume that $s_m \to s$ and $t_m \to t$ with $s < t$, and fix $u_0 \in L_p(D)$.  We can assume that $s_m < (s + t)/2 < t$ for all $m$.  By the previous paragraph, $U_{a_{0}}( (s + t)/2, s_m)u_0 \to U_{a_{0}}((s + t)/2, s)u_0$ in $L_p(D)$.  Corollary~\ref{cor:p-q} implies that
    \begin{multline*}
        U_{a_{0}}(t_m, s_m)u_0 =U_{a_{0}}(t_m, \tfrac{1}{2}(s + t)) (U_{a_{0}}(\tfrac{1}{2}(s + t), s_m)u_0
        \\
        \to U_{a_{0}}(t, \tfrac{1}{2}(s + t)) (U_{a_{0}}(\tfrac{1}{2}(s + t), s)u_0 =U_{a_{0}}(t, s)u_0,
    \end{multline*}
    where the convergence is in $L_p(D)$, too. \qed
\end{proof}

\begin{proposition}
\label{prop:continuity-in-norm-top}
    Let $1 < p < \infty$ and $a_0\in Y_0$. Then the mapping
    \begin{equation*}
       \big[\, \dot{\Delta} \ni (s, t) \mapsto U_{a_0}(t, s) \in \mathcal{L}(L_p(D)) \,\big]
    \end{equation*}
    is continuous.
\end{proposition}
\begin{proof}
     Let $s_m \to s$ and $t_m \to t$ with $s < t$.  Suppose to the contrary that there are $\epsilon > 0$ and $(u_m)_{m = 1}^{\infty} \subset L_p(D)$, $\norm{u_m}_{L_p(D)} = 1$, such that
    \begin{equation*}
        \norm{U_{a_0}(t_m, s_m) u_m - U(t, s) u_m}_{L_p(D)} \ge \epsilon, \quad m = 1, 2, 3, \dots.
    \end{equation*}
    It follows from Proposition~\ref{prop:higher_order-compactness} that, after possibly taking a subsequence and relabelling, we can assume that $U_{a_0}(t_m, s_m) u_m$ converge to $\tilde{u}$ and $U_{a_0}(t, s) u_m$ converge to $\hat{u}$, both in $L_p(D)$.  For any $v_0 \in L_{p'}(D)$ we have, by Proposition~\ref{prop:dual-Lp-q},
    \begin{multline*}
        \langle (U_{a_0}(t_m, s_m) - U_{a_0}(t, s)) u_m, v_0 \rangle_{L_p(D), L_{p'}(D)}
        \\
        = \langle u_m, (U_{a_0}^{*}(s_m, t_m) - U_{a_0}^{*}(s, t))v_0 \rangle_{L_p(D), L_{p'}(D)}.
    \end{multline*}
    Since $\norm{u_m}_{L_p(D)} = 1$, we conclude from the adjoint equation analog of Lemma~\ref{lm:continuity-in-strong-top} that the above expression converges to zero as $m \to \infty$.  Consequently $\tilde{u} = \hat{u}$, a contradiction. \qed
\end{proof}

\begin{proposition}
  \label{prop:higher_orders-skew_product-continuity}
    Assume, in addition, \eqref{a: compactness of Y }. For $1 < p < \infty$ the mapping
    \begin{equation*}
      \bigl[\, Y_0 \times \dot{\Delta} \times  L_p(D) \ni (a_{0}, s, t, u_0) \mapsto U_{a_{0}}(t, s) u_0 \in L_p(D) \, \bigl]
    \end{equation*}
    is continuous.
\end{proposition}
\begin{proof}
  It follows from~\cite[Props.~2.2.12 and~2.2.13]{Monogr} that, for $2 \le p < \infty$, the mapping
    \begin{equation*}
     \big[\, Y_0 \times \dot{\Delta} \times  L_2(D) \ni (a_{0}, s, t, u_0) \mapsto U_{a_{0}}(t, s) u_0 \in L_p(D) \,\big]
    \end{equation*}
  is continuous, too.

  To conclude the proof it suffices to show that for any $1 < p < 2$ the mapping
  \begin{equation*}
     \big[\, Y_0 \times \dot{\Delta} \times  L_p(D) \ni (a_0, s, t, u_0) \mapsto U_{a_0}(t, s) u_0 \in L_2(D) \,\big]
    \end{equation*}
    is continuous.  Observe that if we have $a_{0,m} \to a_0 \in Y_0$, $s_m \to s$, $t_m \to t$ with $s_m < t_n$ and $s < t$, and $u_{0,m} \to u_0 \in L_p(D)$, then from Proposition~\ref{prop:higher_order-compactness} it follows that, after possibly choosing a subsequence, there is $w \in L_2(D)$  such that $U_{a_{0,m}}(t_m, s_m) u_{0,m} \to w$ in $L_2(D)$.  Consequently, $\langle U_{a_{0,m}}(t_m, s_m) u_{0,m}, v \rangle_{L_2(D)} \to \langle w, v \rangle_{L_2(D)}$ as $m \to \infty$, for any $v \in L_2(D)$.  On the other hand, one has, by Proposition~\ref{prop:dual-Lp-q},
    \begin{equation*}
        \langle U_{a_{0,m}}(t_m, s_m) u_{0,m}, v \rangle_{L_2(D)} = \langle  u_{0,m}, U^{*}_{a_{0,m}}(s_m, t_m) v \rangle_{L_p(D), L_{p'}(D)}.
    \end{equation*}
    As $2 < p' < \infty$, an application of the result already obtained to the adjoint equation yields that $U^{*}_{a_{0,m}}(s_m,t_m) v$ converges, as $m \to \infty$, to $U^{*}_{a_0}(s,t) v$ in $L_{p'}(D)$.  As $u_{0,m}$ converges to $u_0$ in $L_p(D)$, we have that $\langle  u_{0,m}, U^{*}_{a_{0,m}}(s_m,t_m) v \rangle_{L_p(D), L_{p'}(D)}$ converges to $\langle  u_0, U^{*}_{a_0}(s,t) v \rangle_{L_p(D), L_{p'}(D)}$, which is, by Proposition~\ref{prop:dual-Lp-q}, equal to $\langle  U_{a_0}(t,s) u_0,  v \rangle_{L_2(D)}$.  As $v \in L_2(D)$ is arbitrary, we have $w = U_{a_0}(t,s) u_0$. \qed
\end{proof}

\begin{proposition}
  \label{prop:higher_orders-skew_product-norm_continuity}
    Assume, in addition, $\eqref{a: compactness of Y }$. For $1 < p < \infty$ the mapping
    \begin{equation*}
      \bigl[\, Y_{0} \times \dot{\Delta} \ni (a_{0}, s, t) \mapsto U_{a_{0}}(t, s) \in \mathcal{L}(L_p(D)) \, \bigr]
    \end{equation*}
    is continuous.
\end{proposition}
\begin{proof}
     In order not to overburden the notation we assume $s = 0$.

    Let $(a_{0, m})_{m = 1}^{\infty} \subset Y_0$ be a sequence converging to $a_0$ as $m \to \infty$, and let $(t_m)_{m = 1}^{\infty} \subset (0,T]$ be a sequence converging to $t > 0$ as $m \to \infty$.  Suppose to the contrary that $\norm{U_{a_{0,m}}(t_m, 0) -U_{a_{0}}(t, 0)}_{\mathcal{L}(L_p(D))}$ does not converge to $0$, that is, there exist $\epsilon > 0$ and a sequence $(u_m)_{m = 1}^{\infty} \subset L_p(D)$, $\norm{u_m}_{L_p(D)} = 1$ for all $m$, such that
    \begin{equation*}
        \norm{U_{a_{0,m}}(t_m, 0) u_m -U_{a_{0}}(t, 0) u_m}_{L_p(D)} \ge \epsilon
    \end{equation*}
    for all $m$.

    It follows from Proposition~\ref{prop:higher_order-compactness} that, after possibly extracting a subsequence, we can assume that $U_{a_{0,m}}(t_m/2, 0) u_m$ and $U_{a_{0}}(t/2, 0) u_m$ converge, as $m \to \infty$, in the $L_p(D)$\nobreakdash-\hspace{0pt}norm. We claim that both converge to the same $\tilde{u}$.  Indeed, it suffices to check that the difference $(U_{a_{0,m}}(t_m/2, 0) -U_{a_{0}}(t/2, 0)) u_m$ converges to zero in $L_p(D)$, which is, in~light of the equalities
    \begin{multline*}
        \langle (U_{a_{0,m}}(t_m/2, 0) -U_{a_{0}}(t/2, 0)) u_m, v \rangle_{L_p(D), L_{p'}(D)}
        \\
        = \langle u_m, (U^{*}_{a_{0,m}}(0, t_m/2) -U^{*}_{a_{0}}(0, t/2)) v \rangle_{L_p(D), L_{p'}(D)}, \quad v \in L_{p'}(D),
    \end{multline*}
    a consequence of the analog for the adjoint equation of Proposition~\ref{prop:higher_orders-skew_product-continuity}.

    Proposition~\ref{prop:higher_orders-skew_product-continuity} implies that
    \begin{multline*}
        \norm{U_{a_{0,m}}(t_m, 0) u_m -U_{a_{0}}(t, t/2) \tilde{u}}_{L_p(D)}
        \\
        = \norm{U_{a_{0,m}}(t_m, t_m/2) (U_{a_{0,m}}(t_m/2, 0) u_m) -U_{a_{0}}(t, t/2) \tilde{u}}_{L_p(D)} \to 0,
    \end{multline*}
    and
    \begin{multline*}
        \norm{U_{a_{0}}(t, 0) u_m -U_{a_{0}}(t, t/2) \tilde{u}}_{L_p(D)}
        \\
        = \norm{U_{a_{0}}(t, t/2) (U_{a_{0}}(t/2, 0) u_m) - U_{a_{0}}(t, t/2) \tilde{u}}_{L_p(D)}
        \\
        \le \norm{U_{a_{0}}(t, t/2)}_{\mathcal{L}(L_p(D))} \norm{U_{a_{0}}(t/2, 0) u_m - \tilde{u}}_{L_p(D)} \to 0
      \end{multline*}
    therefore $\norm{U_{a_{0,m}}(t_m, 0) u_m -U_{a_{0}}(t, 0) u_m}_{L_p(D)}$, converges to zero, a contradiction. \qed
\end{proof}

\section{Mild Solutions}
\label{sect:mild}
In the present section we assume \eqref{a: boundary regularity}, \eqref{a: boundedness} and that $Y$ as in Definition~\ref{def:Y} is such that its flattening $Y_0$ satisfies \eqref{Ellipticity}.  Occasionally we will assume \eqref{a:only_c_continuous}.

\begin{definition}[Multiplication Operator]
For  $a \in Y$, $1 \le p\le \infty$ and $0\le t\le T$ we define multiplication operator $C^{1}_{a}(t)\colon L_p(D)\to L_p(D)$ as follows
\begin{equation*}
    C_{a}^{1}(t)v=c_1(t,\cdot)v.
\end{equation*}
\end{definition}

The $C^{1}_{a}(t)$ operator is well defined as long as assumption~\eqref{a: boundedness} holds. To be more precise we use a corollary from assumption~\eqref{a: boundedness} on $t$\nobreakdash-\hspace{0pt}sections of $c_1$.

\begin{lemma} [Boundedness of Multiplication Operator]
\label{Boundednes of Multiplication Operator}
The multiplication operator $C_{a}^1(t)$ is linear and bounded uniformly with respect to a.e.\ $0< t< T $ and $a\in Y$.
\end{lemma}
\noindent It should be remarked that the exceptional sets can be different for different $a \in Y$.
\begin{proof}
Let $K$ be the norm bound of $Y$ (see assumption~\ref{Y1}). For any $v\in L_p(D)$ by the H\"older inequality we get
\begin{align*}
    \norm{C_{a}^{1}(t)v}_{L_p(D)} & =\norm{c_{1}(t,\cdot)v}_{L_p(D)} \\
    & \le \norm{c_1(t,\cdot)}_{L_{\infty}(D)}\norm{v}_{L_{p}(D)} \\
    & \le K\norm{v}_{L_{p}(D)}
\end{align*}
where above inequality holds for a.e.\ $0< t< T$, so the operator norm of $C^{1}_{a}(t)$ is bounded a.e.\ by $\norm{c_1(t,\cdot)}_{L_{\infty}(D)}$ what can be bounded uniformly with respect to a.e.\ $0 < t< T$ by virtue of assumption~\eqref{a: boundedness}. Since $Y$ is bounded by $K$ we also have uniform boundedness in $a \in Y$. \qed
\end{proof}

Below we present a series of lemmas to prove the measurability of individual parts of the mild solution. We will make frequent use of the Lemma~\ref{lm:Dunford-Schwartz} in this part of the work, in particular, Remark~\ref{remark:measurable_complement} on measurability will also be useful.

\begin{lemma}
 For any $1 < p < \infty$ and any norm bounded set $E \subset C([-1,T],L_p(D))$ the set
\begin{equation}
\label{measurable_step_1}
    \big\{ \, [\,(0,T) \ni t\mapsto (u \circ \Phi)(t)\in L_p(D) \,] : u \in E, \Phi \in \widetilde{\mathcal{R}} \, \big\}
\end{equation}
is bounded in $L_{\infty}((0,T),L_p(D))$.
\end{lemma}
\begin{proof}
Let $u\in E$ and $\Phi\in \widetilde{\mathcal{R}}$. The mapping $u \, \circ \, \Phi$ is  $(\mathfrak{L}((0,T)), \mathfrak{B}(L_p(D)))$\nobreakdash-\hspace{0pt}measurable, since for any fixed open set $V\subset L_p(D)$ the preimage $u^{-1}[V]$ is open, hence $(u\circ \Phi)^{-1}[V]\in \mathfrak{L}((0,T))$. Moreover, the $L_{\infty}((0,T), L_p(D))$\nobreakdash-\hspace{0pt}norm of the map $u\circ \Phi$ is uniformly bounded with respect to $u$ and $\Phi$ by the same constant as $E$ due to the inequality
\begin{align*}
    \norm{u\circ \Phi}_{L_{\infty}((0,T),L_p(D))} &= \esssup_{t\in (0,T)} {\norm{(u\circ \Phi)(t)}_{L_p(D)}}
    \\ & \le \sup_{t\in [-1,T]} \norm{u(t)}_{L_p(D)}  =  \norm{u}_{C([-1,T],L_p(D))}.
\end{align*}   \qed
\end{proof}

\begin{lemma}
\label{lemma:v_m-boundednes}
 For any $1 < p < \infty$ and any norm bounded set $\tilde{E} \subset L_{\infty}((0,T),L_p(D))$ the set
\begin{equation*}
    \big\{ \, [\,(0,T) \ni t\mapsto C^1_{a}(t) \tilde{u}(t)\in L_p(D) \,] : a\in Y, \tilde{u}\in \tilde{E}\, \big\}
\end{equation*}
is bounded in $L_{\infty}((0,T),L_p(D))$.
\end{lemma}
\begin{proof}
Let $\tilde{u}\in \tilde{E}$ and $a\in Y$. From Lemma~\ref{DS_a} follows that the mapping
\begin{equation}
\label{measurable_step_2}
    \big[ \, (0,T)\times D \ni (t,x)\mapsto \tilde{u}(t)[x]\in \RR \,\big]
\end{equation}
is $(\mathfrak{L}((0, T))\otimes\mathfrak{L}(D), \mathfrak{B}(\RR))$\nobreakdash-\hspace{0pt}measurable. Hence the function
\begin{equation}
\label{measurable_step_3}
    \big[ \, (0,T)\times D \ni (t,x)\mapsto \big( C_a(t)\tilde{u}(t)\big) [x]\in \RR \,\big]
\end{equation}
for any $a\in Y$ is $(\mathfrak{L}((0, T))\otimes\mathfrak{L}(D), \mathfrak{B}(\RR))$\nobreakdash-\hspace{0pt}measurable, since it can be rewritten as the product of $(\mathfrak{L}((0, T))\otimes\mathfrak{L}(D), \mathfrak{B}(\RR))$\nobreakdash-\hspace{0pt}measurable functions, namely
\begin{equation}
    \big[ \, (0,T)\times D \ni (t,x)\mapsto  c_1(t,x) \, \tilde{u}(t) [x]\in \RR \,\big].
\end{equation}
It suffices now to notice that for a.e.\ $t \in (0,T)$ the $t$\nobreakdash-\hspace{0pt}section of \eqref{measurable_step_3} belongs (by the definition of the multiplication operator) to $L_p(D)$. So from Lemma~\ref{DS_b} it follows that the mapping
\begin{equation*}
    \big[ \, (0,T) \ni t \mapsto  C_a(t)\tilde{u}(t) \in L_p(D) \,\big]
\end{equation*}
is $(\mathfrak{L}((0, T)), \mathfrak{B}(L_p(D)))$\nobreakdash-\hspace{0pt}measurable. By the norm estimate,
\begin{align*}
      \norm{C^1_{a}(\cdot)\tilde{u}}_{L_{\infty}((0,T),L_p(D))}
      &= \esssup_{t\in(0,T)}{\norm{c_1(t,\cdot) \tilde{u}(t)}_{L_p(D)}} \\
      &\le K \norm{\tilde{u}}_{L_{\infty}((0,T), L_p(D))}
\end{align*}
we obtain the statement.  \qed
\end{proof}

\begin{lemma}
\label{lm:I_i}
    Assume $1 < p < \infty$, $a_0 \in Y_{0}$, and $u \in L_{\infty}((0, T), L_p(D))$. Then
    \begin{enumerate}[label=\textup{(\roman*\textup)}, ref=\ref{lm:I_i}\textup{(\roman*)}]
    \item\label{lm:I_i_(i)}
    for any $0 < t \le T$ the function
    \begin{equation}
    \label{eq:I_i}
    [\, \zeta \mapsto U_{a_0}(t, \zeta) u(\zeta), \text{ for a.e. } \zeta \in (0, t) \,]
    \end{equation}
    belongs to $L_{\infty}((0, t), L_{p}(D))$
    moreover, the linear operator assigning~\eqref{eq:I_i} to $u$ belongs to $\mathcal{L}(L_{\infty}((0,1), L_p(D)), L_{\infty}((0,1), L_p(D)))$, with the norm bounded uniformly in $a_0 \in Y_0$,
        \item\label{lm:I_i_(ii)}
        the mapping
        \begin{equation}
        \label{eq:I_i-1}
            \biggl[\, [0, T] \ni t \mapsto \int_{0}^{t} U_{a_0}(t, \zeta) u(\zeta) \, \mathrm{d}\zeta \,\biggr]
        \end{equation}
        belongs to $C([0, T], L_p(D))$.
    \end{enumerate}
\end{lemma}
\begin{proof}
      Fix $0 < t \le T$.  We show first that \eqref{eq:I_i} defines a $(\mathfrak{L}((0, t)), \mathfrak{B}(L_{p}(D)))$\nobreakdash-\hspace{0pt}measurable function.  It is equivalent, by Theorem~\ref{thm:equiv-measurable}, to showing that for each $v \in L_{p'}(D)$ the function
        \begin{equation*}
            [\, \zeta \mapsto \langle U_{a_0}(t, \zeta) u(\zeta), v \rangle_{L_{p}(D), L_{p'}(D)}]
        \end{equation*}
    is $(\mathfrak{L}((0, t)), \mathfrak{B}(\RR))$\nobreakdash-\hspace{0pt}measurable.  By Proposition~\ref{prop:dual-Lp-q}, for Lebesgue\nobreakdash-\hspace{0pt}a.e.\ $\zeta \in [0, t)$ there holds
        \begin{equation*}
            \langle U_{a_0}(t, \zeta) u(\zeta), v \rangle_{L_{p}(D), L_{p'}(D)} = \langle u(\zeta),U_{a_0}^{*}(\zeta, t) v \rangle_{L_p(D), L_{p'}(D)}
        \end{equation*}
    It suffices now to notice that $u$ is $(\mathfrak{L}((0, t)), \mathfrak{B}(L_p(D)))$\nobreakdash-\hspace{0pt}measurable, by assumption, and that the function
        \begin{equation*}
            \bigl[\, [0, t) \ni \zeta \mapsto U_{a_0}^{*}(\zeta, t) v \in L_{p'}(D) \,\bigr]
        \end{equation*}
    is continuous, by the adjoint equation analog of Lemma~\ref{lm:continuity-in-strong-top}. It follows from Proposition~\ref{prop:higher_orders-skew_product-p-q-estimates_2} that the function
     \begin{equation*}
         \big[\, (0, t) \ni \zeta \mapsto \norm{U_{a_0}(t, \zeta)}_{\mathcal{L}(L_p(D))} \,\big]
     \end{equation*}
     belongs to $L_{\infty}((0, t))$.  Then the membership of~\eqref{eq:I_i} in $L_{\infty}((0, t), L_{p}(D))$ as well as the bound on its norm follow from the generalized H\"older inequality.  The proof of part (i) is thus completed.

    \smallskip
    We proceed to the proof of part (ii).  By part (i), the function~\eqref{eq:I_i-1} is well defined.  Let $0 \le t_1 \le t_2 \le T$.  We write
    \begin{multline*}
        \int_{0}^{t_2} U_{a_0}(t_2, {\zeta}) u({\zeta}) \, \mathrm{d}{\zeta} - \int_{0}^{t_1} U_{a_0}(t_1, {\zeta}) u({\zeta}) \, \mathrm{d}{\zeta}
        \\
        = \int_{0}^{t_1} \bigl( U_{a_0}(t_2, {\zeta}) - U{a_0}(t_1, {\zeta})\bigr) u({\zeta}) \, \mathrm{d}{\zeta} + \int_{t_1}^{t_2} \U_{a_0}(t_2, {\zeta}) u({\zeta}) \, \mathrm{d}{\zeta}.
    \end{multline*}
    Let $\epsilon > 0$.  As \eqref{eq:I_i} belongs to $L_{\infty}((0, t), L_{p}(D))$, it is a consequence of~\cite[Thm.~II.2.4(i)]{DiUhl} that the $L_p(D)$\nobreakdash-\hspace{0pt}norm of the second term on the right-hand side can be made $< \epsilon/3$ by taking $t_1, t_2$ sufficiently close to each other.  Regarding the first term, we write
    \begin{multline*}
        \int_{0}^{t_1} \bigl( \U_{a_0}(t_2, {\zeta}) - \U_{a_0}(t_1, {\zeta})\bigr) u({\zeta}) \, \mathrm{d}{\zeta}
        \\
        = \int_{0}^{t_1 - \eta} \bigl( \U_{a_0}(t_2, {\zeta}) - \U_{a_0}(t_1, {\zeta})\bigr) u({\zeta}) \, \mathrm{d}{\zeta} + \int_{t_1 - \eta}^{t_1} \bigl( \U_{a_0}(t_2, {\zeta}) - \U_{a_0}(t_1, {\zeta})\bigr) u({\zeta}) \, \mathrm{d}{\zeta}.
    \end{multline*}
    Again by \cite[Thm.~II.2.4(i)]{DiUhl}, for $\eta > 0$ sufficiently small there holds
    \begin{equation*}
            \biggl\lVert \int_{t_1 - \eta}^{t_1} \bigl( \U_{a_0}(t_2, {\zeta}) - \U_{a_0}(t_1, {\zeta})\bigr) u({\zeta}) \, \mathrm{d}{\zeta} \biggr\rVert_{L_p(D)} < \frac{\epsilon}{3}.
    \end{equation*}
    It follows from Proposition~\ref{prop:continuity-in-norm-top} that the assignment
    \begin{equation*}
        \bigl[\, \{\, \dot{\Delta} : \eta \le \zeta + \eta \le t \le T \,\} \ni (\zeta, t) \mapsto U_{a_0}(t, \zeta) \in \mathcal{L}(L_p(D)) \,\bigr]
   \end{equation*}
    is uniformly continuous, consequently there exists $\delta > 0$ such that if $\eta \le \zeta + \eta \le t_1 \le t_2$, $t_2 - t_1 < \delta$, then
    \begin{equation*}
        \lVert U_{a_0}(t_2, \zeta) - U_{a_0}(t_1, \zeta) \rVert_{\mathcal{L}(L_p(D))}  < \frac{\epsilon}{3 \norm{u}_{L_1((0, T), L_p(D))}}.
    \end{equation*}
   Therefore
    \begin{equation*}
        \biggl\lVert \int_{0}^{t_1 - \eta} \bigl( U_{a_0}(t_2, {\zeta}) -U_{a_0}(t_1, {\zeta})\bigr) u({\zeta}) \, \mathrm{d}{\zeta} \biggr\rVert_{L_p(D)}
         < \frac{\epsilon}{3}.
    \end{equation*}
    This concludes the proof of part (ii).  \qed
\end{proof}

\begin{definition}[Mild Solution]
For $1 \le p<\infty$, $a \in Y$, $0 \le s < T_0 \le T$ and $u_0 \in C([s-1,s], L_p(D))$ and $R \in \mathcal{R}$ the function $u\in C([s-1,T_0], L_p(D))$ such that
  \begin{equation}
  \label{mild_initial_condi}
    u(t) =  u_0(t) \text{ for } t \in [s- 1, s],
  \end{equation}
holds and the integral equation
\begin{equation}
\label{mild_equation_semip}
    u(t)=\U_{\tilde{a}}(t,s) u_0(s) +\int_{s}^{t} \U_{\tilde{a}} (t,\zeta) C^{1}_{a}(\zeta)u(\zeta-R(\zeta)) \,\dd\zeta
\end{equation}
is satisfied in $L_p(D)$ on $[s,T_0]$ will be called a \emph{mild solution} of \eqref{main-eq}$_a+$\eqref{main-bc}$_{a}$.

For $T_0 = T$ we have a \emph{global mild solution}.
\end{definition}

At first note that the concept of mild solution, especially part~\eqref{mild_equation_semip}, is well defined based on Lemma~\ref{lemma:v_m-boundednes} and Lemma~\ref{lm:I_i}.
At some moments we use the name ``mild solution'' to describe function $u\upharpoonright_{[0,T_0]}$ instead of $u\in C([-1,T_0],L_p(D))$ satisfying \eqref{mild_initial_condi} and \eqref{mild_equation_semip}. This convention seems more natural especially in the context of continuous dependence on coefficients. A similar convention can be found in the literature~\cite{Hale-Le}.

\subsection{Existence and Uniqueness of Global Mild Solutions}

\begin{proposition}
\label{pro:local_uniqueness}
There exists $\Theta_0 \in (0,T]$ such that for any $1<p<\infty$, $a \in Y $, $R\in \mathcal{R}$, $0\le s\le T-\Theta_0$, $u_0\in C([s-1,s],L_p(D))$ and any $0<\Theta\le \Theta_0$ there exist unique solution of \eqref{main-eq}$_{a}$\textup{+}%
\eqref{main-bc}$_{a}$ on $[s-1,s+\Theta]$ with initial condition $u_0$.
\end{proposition}

\begin{proof}
The idea of the proof runs as~follows.  The solution is obtained as a fixed point of the contraction mapping $\mathfrak{G}$ of $C([s,s+\Theta], L_p(D))$ into itself (see \ref{prop:higher_orders-skew_product-p_ii}, \ref{lm:I_i_(ii)} and \ref{lemma:v_m-boundednes}) defined as
\begin{equation}
\label{eq:Gothic-G}
  \begin{aligned}
  (\mathfrak{G}u)[t] \vcentcolon= {} & \U_{\tilde{a}} (t,s)u_0(s) +\int_{s}^{t} \U_{\tilde{a}} (t,\zeta) C^{1}_{a}(\zeta)u(\zeta-R(\zeta)) \,\dd\zeta ,
  \end{aligned}
\end{equation}
where $s \le t \le s+\Theta$ and $\Theta \in (0, T]$ is sufficiently small. Until revoking, $u$ and $v$ stand for  generic functions in $C([s,s+\Theta], L_p(D))$. For such a $u$ we interpret $u(\zeta - R(\zeta))$ (similarly $v$) as $u_0(\zeta - R(\zeta))$ when $\zeta-R(\zeta) \in (-1, 0)$.
\begin{align*}
    \norm{(\mathfrak{G}u)[t] - (\mathfrak{G}v)[t] }_{L_p(D)} &\le \int_{s}^{t} \norm{\U_{\tilde{a}} (t,\zeta)C^{1}_{a}(\zeta)\big( u(\zeta-R(\zeta))-v(\zeta-R(\zeta)) \big) }_{L_p(D)}\, \dd\zeta  \\
    & \le MKe^{\gamma t}  \int_{s}^{t} \norm{ u(\zeta-R(\zeta))-v(\zeta-R(\zeta))  }_{L_p(D)}\, \dd\zeta  \\
    &\le  MKe^{\gamma t} \int_{s}^{t} \sup_{s\le \xi \le t} \norm{ u(\xi-R(\xi))-v(\xi-R(\xi))  }_{L_p(D)} \, \dd\zeta  \\
    &\le  MKe^{\gamma t} \int_{s}^{t} \sup_{s-1\le \xi \le t} \norm{ u(\xi)-v(\xi)  }_{L_p(D)} \, \dd\zeta  \\
    &= MKe^{\gamma t} \int_{s}^{t} \sup_{s\le \xi\le t} \norm{ u(\xi)-v(\xi)  }_{L_p(D)} \, \dd\zeta  \\
    &\le MKe^{\gamma T}\Theta \norm{u-v}_{C([s,s+\Theta],L_p(D))}.
\end{align*}
By taking $\displaystyle 0< \Theta \le \Theta_0 \coloneqq 1/(2MK e^{\gamma T})$ we obtain that the contraction coefficient is less than $1$. \qed
\end{proof}
The Contraction Mapping Principle guarantees the existence and uniqueness of the fixed point $u$ of  $\mathfrak{G}$, which is then the unique mild $L_p$\nobreakdash-\hspace{0pt}solution of \eqref{main-eq}$_{a}$\textup{+}%
\eqref{main-bc}$_{a}$ on $[s-1, \Theta_0]$ satisfying the initial condition \eqref{main-ic}.

\begin{lemma}
\label{lm:gluing}
For any $1<p<\infty$, $0<s_1<s_2\le T$, $a\in Y$, $R\in\mathcal{R}$, $u_0\in C([-1,0],L_p(D))$ and $v \colon [-1,s_2]\to L_p(D)$ the following statements are equivalent:
\begin{enumerate}[label=\textup{(\roman*\textup)}, ref=\textup{(\roman*)}]
\item\label{glue} a function $v$ is the mild solution of \eqref{main-eq}$_{a}$\textup{+}%
\eqref{main-bc}$_{a}$ on $[-1,s_2]$ with initial condition $u_0$,
    \item\label{unglue} a function $v\!\!\restriction_{[-1,s_1]}$ is the mild solution of \eqref{main-eq}$_{a}$\textup{+}%
\eqref{main-bc}$_{a}$ with initial condition $u_0$ and $v\!\!\restriction_{[s_1-1,s_2]}$ is the mild solution of \eqref{main-eq}$_{a}$\textup{+}%
\eqref{main-bc}$_{a}$ with initial condition $v\!\!\restriction_{[s_1-1,s_1]}$.
\end{enumerate}

\end{lemma}
\begin{proof}
    Let $p,s_1,s_2,a,R,u_0$ and $v$ be as in the statement. To prove $\ref{glue} \Rightarrow\ref{unglue}$ it suffices to see that for any  $s_1 \le t\le s_2$, in view of \eqref{eq:cocycle2-2} and \cite[Lemma~11.45]{AliB} there holds
    \begin{equation}
    \label{eq:gluing}
        \begin{aligned}
         v(t) = {} & U_{\tilde{a}}(t, 0) u_0(0) + \int\limits_{0}^{t} U_{\tilde{a}}(t, \zeta) C^1_{a}(\zeta) v(\zeta-R(\zeta)) \, \mathrm{d}\zeta \\
        ={} & U_{\tilde{a}}(t, s_1)\Bigg(  U_{\tilde{a}}(s_1, 0)  u_0(0) + \int\limits_{0}^{s_1} U_{\tilde{a}}(s_1, \zeta) C^1_{a}(\zeta) v(\zeta-R(\zeta)) \, \mathrm{d}\zeta \Bigg) \\
         {} & + \int\limits_{s_1}^{t} U_{\tilde{a}}(t, \zeta) C^1_{a}(\zeta) v(\zeta-R(\zeta)) \, \mathrm{d}\zeta\\
        = {}  & U_{\tilde{a}}(t, s_1)v{\restriction_{[s_1-1,s_1]}}(s_1)+ \int\limits_{s_1}^{t} U_{\tilde{a}}(t, \zeta) C^1_{a}(\zeta) v{\restriction_{[s_1-1,s_2]}}(\zeta-R(\zeta)) \, \mathrm{d}\zeta.
        \end{aligned}
    \end{equation}
    In order to prove $\ref{unglue}\Rightarrow \ref{glue}$ fix $t\in [-1,s_2]$ and consider the cases: $t\in[-1,0]$, $t\in[0,s_1]$ or $t\in[s_1,s_2]$. Since the first two cases are straightforward and the third is a similar calculation to~\eqref{eq:gluing}, the proof is finished.  \qed
\end{proof}

\begin{theorem}
\label{thm:mild_solution_existence}
For any $1 < p < \infty$, $a \in Y$, $u_0 \in C([-1,0],L_p(D))$ and $R \in \mathcal{R}$ equation \eqref{main-eq}$_{a}$\textup{+}%
\eqref{main-bc}$_{a}$ has a unique global mild solution on $[0,T]$.
\end{theorem}

\begin{proof}
    Fix $a$, $u_0$ and $R$ as in the statement. Let
    \begin{equation*}
        Q=\{q\in [0,T]: \text{\eqref{main-eq}$_{a}$\textup{+}%
\eqref{main-bc}$_{a}$ has a unique mild solution on }[-1,q] \}.
    \end{equation*}
    It suffices to prove that $T\in Q$. Suppose to the contrary that $T\not\in Q$. Since $\Theta_0 \in Q$ (where $\Theta_0$ stands for constant obtained in Proposition~\ref{pro:local_uniqueness}) and $Q \subset [0,T]$, $\sup Q < \infty$.  It is straightforward that $0 \le \sup Q-\Theta_0/2<\sup Q\le T$ hence there exists $s \in Q$ such that $s > \sup Q - \Theta_0/2$.

    Let $v_1 \colon [-1,s] \to L_p(D)$ be the unique mild solution with initial condition $u_0$. From the definition of $\Theta$ it follows that there is a mild solution $v_2 \colon [s-1,\min\{s+\Theta_0,T\}] \to L_p(D)$ with initial condition $v_1\!\!\restriction_{[s-1,s]}$. Let
    \begin{equation*}
        v(t) =
      \begin{cases}
       v_1(t) & \text{for } t \in [-1, s]
        \\
       v_2(t) & \text{for } t \in [s-1, \min\{s+\Theta_0,T\} ].
      \end{cases}
    \end{equation*}
 We claim that $v$ is a unique mild solution of \eqref{main-eq}$_{a}$\textup{+}%
\eqref{main-bc}$_{a}$ on $[-1,\min\{s+\Theta_0,T\}]$ with initial condition $u_0$.  From Lemma~\ref{lm:gluing} it follows that $v$ is in fact a mild solution.  For uniqueness, assume $w \colon [-1, \min \{ s+\Theta_0, T \}]\to L_p(D)$ is any mild solution. Then clearly $w\!\!\restriction_{[-1, s]} = v_1$.  Moreover, by Lemma~\ref{lm:gluing}, the function $w\!\!\restriction_{[s-1, \min \{ s+\Theta_0, T \}]}$ is a mild solution with initial condition $w\!\!\restriction_{[s-1, s]} = v_1\!\!\restriction_{[s-1, s]}$, so by the uniqueness of $v_2$ we have that $w\!\!\restriction_{[s-1, \min \{ s+\Theta_0, T \}]} = v_2$. Hence $v = w$.  The proof is completed by the following observation: if $\min\{s+\Theta_0,T\}=s+\Theta_0$ then $s+\Theta_0 \in Q$, so we get a contradiction with the fact that $s+\Theta_0 > \sup Q$: otherwise $T\in Q$, which contradicts the assumption.  \qed
\end{proof}

The above result allows us to define a mild solution of \eqref{main-eq}$_{a}$\textup{+}%
\eqref{main-bc}$_{a}$ on the whole of $[-1, T]$ or $[s-1,T]$ if necessary.

For $s = 0$, to stress the dependence of the solution on $a$, $u_0$, $R$ we write $u(\cdot; a, u_0,R)$. For $t \in [-1, 0]$, $u(t; a, u_0, R)$ is interpreted as $u_0(t)$. Moreover, when it does not lead to confusion, we sometimes write $u(t;a, u_0, \Phi)$ instead of $u(t; a, u_0, R)$.

\subsection{Compactness of Solution Operator}
\label{subsect:compactness}

\begin{lemma}
\label{lm:compact-aux}
    Assume $1 < p< \infty$ and $0 < T_1 \le T$.
    Then for any bounded $F \subset L_{\infty}((0, T), L_p(D))$ the set
    \begin{equation*}
        \widehat{F} \vcentcolon= \biggl\{ \int\limits_{0}^{t} U_{a_0}(t, {\zeta}) u({\zeta}) \, \mathrm{d}{\zeta} : a_0 \in Y_{0}, \ u \in F, \ t \in [T_1, T] \biggr\}
    \end{equation*}
    is precompact in $L_p(D)$.
\end{lemma}
\begin{proof}
    Compare~\cite[Thm.~6.1.3]{Monogr}.  Fix $p$, $T_1$ and $F$ as in the statement. Let $(t_{m})_{m=1}^{\infty} \subset [T_1,T]$, $(a_{0,m})_{m=1}^{\infty} \subset Y_{0}$,  $(u_{m})_{m=1}^{\infty} \subset F$.   We claim that for any fixed $l \in \NN$ the set
    \begin{equation*}
        \begin{aligned}
            \widetilde{F}_{l} & \vcentcolon= \biggl\{ \int\limits_{0}^{t_{m}-\frac{1}{l}} U_{a_{0,m}}(t_{m}, {\zeta}) u_{m}({\zeta}) \, \mathrm{d}\zeta : {m} \in \NN \biggr\}.
        \end{aligned}
    \end{equation*}
    is precompact in $L_{p}(D)$. Denote by $M_0 > 0$ the supremum of the $L_{\infty}((0, T), \allowbreak L_{p}(D))$\nobreakdash-\hspace{0pt}norms of $u_m$, and put $\widecheck{F}_{l}$ to be the closure in $L_{p}(D)$ of the set
    \begin{equation*}
        \left\{U_{a_0}(s, 0) {\tilde{u}}: a_0 \in Y_{0}, s\in \Bigl[  \frac{1}{l},T\Bigr], \rVert \tilde{u} \lVert_{L_p(D)} \le M_0 \right\}
    \end{equation*}
    The set $\widecheck{F}_{l}$ is balanced.  We have $\widecheck{F}_{l} \subset \widecheck{F}_{l+1}$.  By Proposition~\ref{prop:higher_order-compactness}, $\widecheck{F}_{l}$ is compact.  \cite[Cor.~II.2.8]{DiUhl} implies that
    \begin{equation*}
        \int\limits_{0}^{t_{m}-\frac{1}{l}} U_{a_{0,m}}(t_{m}, \zeta) u_m(\zeta) \, \mathrm{d}\zeta \in T \cdot \conv{\widecheck{F}_{l}}
    \end{equation*}
    where $\conv$ denotes the closed convex hull in $L_p(D)$. As, by Mazur's theorem (\cite[Thm.~II.2.12]{DiUhl}), $T \cdot \conv{\widecheck{F}_{l}}$ is compact for any $l \in \NN$, this proves our claim that $\widetilde{F}_{l}$ are precompact in $L_{p}(D)$.  By a diagonal process we can assume without loss of generality that for each $l \in \NN$ the integrals
    \begin{equation*}
        \int\limits_{0}^{{t_m}-\frac{1}{l}} U_{a_{0,m}}(t_{m}, \zeta) u_{m}(\zeta) \, \mathrm{d}\zeta
    \end{equation*}
    converge, as $m \to \infty$, in $L_{p}(D)$.

    Lemma~\ref{lm:I_i_(i)} guarantees that the functions $U_{a_{0,m}}(t_m, \cdot) u_m(\cdot)$
    belong to $L_{\infty}((0, t_m), L_p(D))$, with their $L_{\infty}((0, t_m), L_p(D))$\nobreakdash-\hspace{0pt}norms bounded uniformly in $m$.  We estimate, via H\"older's inequality,

\begin{multline*}
    \biggl\lVert \int_{t_{m}-\frac{1}{l}}^{t_m} U_{a_{0,m}}(t_m, \zeta) u_m(\zeta) \, \mathrm{d}{\zeta} \biggr\rVert_{L_p(D)}
       \\
     \le \int_{t_{m}-\frac{1}{l}}^{t_m} \lVert U_{a_{0,m}}(t_m, \zeta) u_m(\zeta) \rVert_{L_p(D)} \, \mathrm{d}{\zeta}
     \\
     \le \norm{U_{a_{0,m}}(t_m, \cdot) u_m(\cdot)}_{L_{\infty}((0, t_{m}), L_p(D))} \cdot (1/l).
\end{multline*}

   It follows from Proposition~\ref{prop:higher_orders-skew_product-p-q-estimates_2} that for any $\epsilon > 0$ there is $l_0 \in \NN$ such that
    \begin{equation*}
        \biggl\lVert \, \int\limits_{t_{m}-\frac{1}{l_0}}^{t_m} U_{a_{0,m}}(t_{m}, \zeta) u_{m}(\zeta) \, \mathrm{d}\zeta  \, \biggr\rVert_{L_{p}(D)} < \frac{\epsilon}{3}
    \end{equation*}
    uniformly in $m \in \NN$.  By the previous paragraph, there is $m_0$ such that if $m_1, m_2 \ge m_0$ then
    \begin{equation*}
        \biggl\lVert \int\limits_{0}^{t_{{m}_1} - \frac{1}{l_0}} U_{a_{0,m_1}}(t_{{m}_1}, \zeta) u_{{m}_1}(\zeta) \, \mathrm{d}\zeta - \int\limits_{0}^{t_{{m}_2} - \frac{1}{l_0}} U_{a_{0,m_2}}(t_{{m}_2}, \zeta) u_{{m}_2}(\zeta) \, \mathrm{d}\zeta \,
        \biggr\rVert_{L_{p}(D)} < \frac{\epsilon}{3}.
    \end{equation*}
    Therefore
    \begin{equation*}
        \biggl\lVert \int\limits_{0}^{t_{{m}_1}} U_{a_{0,m_1}}(t_{{m}_1}, \zeta) u_{{m}_1}(\zeta) \, \mathrm{d}\zeta - \int\limits_{0}^{t_{{m}_2}} U_{a_{0,m_2}}(t_{{m}_2}, \zeta) u_{{m}_2}(\zeta) \, \mathrm{d}\zeta \,
        \biggr\rVert_{L_{p}(D)} < \epsilon
    \end{equation*}
    for any $m_1, m_2 \ge m_0$.

    From this it follows that
    \begin{equation*}
        \biggl( \int\limits_{0}^{t_{m}} \widetilde{U}_{a_{0,m}}(t_{m}, \zeta) u_{m}(\zeta) \, \mathrm{d}\zeta \biggr)_{\!\! m = 1}^{\!\! \infty}
    \end{equation*}
    is a Cauchy sequence in $L_{p}(D)$.  Therefore $\widehat{F}$ is precompact in $L_{p}(D)$. \qed
\end{proof}

\begin{lemma}
\label{lm:compact-aux-1}
    For any $1 < p < \infty$ and any bounded $F \subset L_{\infty}((0, T), L_p(D))$ the set
    \begin{equation*}
        \biggl\{ \Bigl[\, [0, T] \ni t \mapsto \int\limits_{0}^{t} U_{a_0}(t, \zeta) u(\zeta) \, \mathrm{d}\zeta \in L_{p}(D) \,\Bigr]: a_0 \in Y_{0}, \ u \in F \biggr\}
    \end{equation*}
    is precompact in $C([0, T], L_{p}(D))$.
\end{lemma}
\begin{proof}
    By the Ascoli--Arzel\`a theorem, it suffices, taking Lemma~\ref{lm:compact-aux} into account, to show that for any $\epsilon > 0$ there is $\delta > 0$ such that, if $0 \le t_1 \le t_2 \le T$, $t_2 - t_1 < \delta$, then
    \begin{equation*}
        \biggl\lVert \int_{0}^{t_2} U_{a_0}(t_2, {\zeta}) u({\zeta}) \, \mathrm{d}{\zeta} - \int_{0}^{t_1} U_{a_0}(t_1, {\zeta}) u({\zeta}) \, \mathrm{d}\zeta \biggr\rVert_{L_{p}(D)} < \epsilon
    \end{equation*}
    for all $a_0 \in Y_{0}$ and all $u \in F$.  In order not to introduce too many constants we assume that $F$ equals the unit ball in $L_{\infty}((0, T), L_p(D))$.

    We write
    \begin{multline*}
        \int_{0}^{t_2} U_{a_0}(t_2, {\zeta}) u({\zeta}) \, \mathrm{d}{\zeta} - \int_{0}^{t_1} U_{a_0}(t_1, {\zeta}) u({\zeta}) \, \mathrm{d}{\zeta}
        \\
        = \int_{0}^{t_1} \bigl( U_{a_0}(t_2, {\zeta}) - U_{a_0}(t_1, {\zeta})\bigr) u({\zeta}) \, \mathrm{d}{\zeta} + \int_{t_1}^{t_2} U_{a_0}(t_2, {\zeta}) u({\zeta}) \, \mathrm{d}{\zeta}
    \end{multline*}
    By Proposition~\ref{prop:higher_orders-skew_product-p-q-estimates_2},
    \begin{equation}
    \label{eq:compact-aux-1}
        \biggl\lVert \int_{t_1}^{t_2} U_{a_0}(t_2, {\zeta}) u({\zeta}) \, \mathrm{d}{\zeta} \biggr\rVert_{L_{p}(D)} < \frac{\epsilon}{3},
    \end{equation}
    provided $t_2 - t_1 < \epsilon/(3 M e^{\gamma T})$.

    Further, we write
    \begin{multline*}
        \int_{0}^{t_1} \bigl( U_{a_0}(t_2, {\zeta}) - U_{a_0}(t_1, {\zeta})\bigr) u({\zeta}) \, \mathrm{d}{\zeta}
        \\
        = \int_{0}^{t_1 - \eta} \bigl( U_{a_0}(t_2, {\zeta}) - U_{a_0}(t_1, {\zeta})\bigr) u({\zeta}) \, \mathrm{d}{\zeta} + \int_{t_1 - \eta}^{t_1} \bigl( U_{a_0}(t_2, {\zeta}) - U_{a_0}(t_1, {\zeta})\bigr) u({\zeta}) \, \mathrm{d}{\zeta}.
    \end{multline*}
    By Proposition~\ref{prop:higher_orders-skew_product-p-q-estimates_2}, if $0 < \eta < \epsilon/(6M e^{\gamma T})$ then
    \begin{equation}
    \label{eq:compact-aux-2}
        \biggl\lVert \int_{t_1 - \eta}^{t_1} \bigl( U_{a_0}(t_2, {\zeta}) - U_{a_0}(t_1, {\zeta})\bigr) u({\zeta}) \, \mathrm{d}{\zeta} \biggr\rVert_{L_{q}(D)} < \frac{\epsilon}{3}.
    \end{equation}
    It follows from Proposition~\ref{prop:higher_orders-skew_product-norm_continuity} that the assignment
    \begin{equation*}
        \bigl[\, Y_0 \times [\eta, T] \ni (a_0, t) \mapsto U_{a_0}(t, 0) \in \mathcal{L}(L_p(D)) \,\bigr]
    \end{equation*}
    is uniformly continuous, consequently there exists $\delta > 0$ such that if $\eta \le s_1 < s_2 $, $s_2 - s_1 < \delta$, then
    \begin{equation*}
        \lVert U_{a_0}(s_2, 0) - U_{a_0}(s_1, 0) \rVert_{\mathcal{L}(L_p(D))}  < \frac{\epsilon}{3 T}.
    \end{equation*}
   Therefore
    \begin{equation*}
    \label{eq:compact-aux-3}
        \biggl\lVert \int_{0}^{t_1 - \eta} \bigl( U_{a_0}(t_2, {\zeta}) - U_{a_0}(t_1, {\zeta})\bigr) u({\zeta}) \, \mathrm{d}{\zeta} \biggr\rVert_{L_{q}(D)}
        < \frac{\epsilon}{3}.
    \end{equation*}
    The estimates~\eqref{eq:compact-aux-1}, \eqref{eq:compact-aux-2} and~\eqref{eq:compact-aux-3} do not depend on the choice of $a_0 \in Y_0$, so gathering them gives the required property.  \qed
\end{proof}

\begin{theorem}
  \label{thm:delay-compactness-X_2}
  For any $0 < T_1 \le T$, any $1 < p < \infty$ and any bounded $E \subset C([- 1, 0], L_p(D))$ the set
  \begin{equation*}
    \bigl\{\, \bigl[\, [T_1, T] \ni t \mapsto u(t; a, u_0,R) \,\bigr] : a \in Y, u_0 \in E, R\in \mathcal{R} \,\bigr\}
  \end{equation*}
  is precompact in $C([T_1, T], L_p(D))$.
\end{theorem}
\begin{proof}
    We will use the notation $I_i(t; a, u_0,R)$, $i = 0, 1$ where

    \begin{equation}
\label{eq:I}
    \begin{aligned}
       I_0(t) & \vcentcolon =  U_{\tilde{a}}(t, 0) u_0(0),
       \\[1ex]
       I_1(t) & \vcentcolon = \int\limits_{0}^{t} U_{\tilde{a}}(t, \zeta) C^1_{a}(\zeta) u(\zeta-R(\zeta)) \, \dd\zeta ,
    \end{aligned}
\end{equation}
     taking account of the parameter $a$ and the initial value $u_0$.  The precompactness of the set
    \begin{equation*}
        \bigl\{\, \bigl[\, [T_1, T] \ni t \mapsto I_0(t; a, u_0,R) \,\bigr] : a \in Y, u_0 \in E, R\in \mathcal{R} \,\bigr\}
    \end{equation*}
    in $C([T_1, T], L_p(D))$ is a consequence of Proposition~\ref{prop:higher_order-compactness}.
    %In view of Proposition~\ref{prop:delay-compactness-X_2}
    In order to prove the precompactness in $C([T_1, T], L_p(D))$ of
    \begin{equation*}
        \bigl\{\, \bigl[\, [T_1, T] \ni t \mapsto I_1(t; a, u_0,R) \,\bigr] : a \in Y, u_0 \in E, R\in \mathcal{R} \,\bigr\},
    \end{equation*}
    it suffices to use results from Lemma~\ref{lemma:v_m-boundednes} and Lemma~\ref{lm:compact-aux-1}.  \qed
\end{proof}

Theorem~\ref{thm:delay-compactness-X_2} leads to the following conclusion about precompactness of the solutions up to zero. Since under additional assumption~\eqref{a:only_c_continuous} for a fixed $u_0 \in C([-1,0],L_p(D))$ the set
    \begin{equation*}
        \bigl\{\, \bigl[\, [0, T] \ni t \mapsto U_{\tilde{a}}(t, 0) u_0(0) \,\bigr] : a \in Y \,\bigr\}
    \end{equation*}
is simply a singleton, this observation combined with Lemma~\ref{lm:compact-aux-1} leads to the following result.
\begin{theorem}
\label{c:delay-compactness-up_to_0}
 Assume additionally~\eqref{a:only_c_continuous}. For any $1 < p < \infty$ and any $u_0 \in C([- 1, 0], L_p(D))$ the set
  \begin{equation*}
    \bigl\{\, \bigl[\, [0, T] \ni t \mapsto u(t; a, u_0,R) \,\bigr] : a \in Y, R\in \mathcal{R} \,\bigr\}
  \end{equation*}
  is precompact in $C([0, T], L_p(D))$.
\end{theorem}

\section{Continuous Dependence on Initial Conditions}
\label{sect:dependence_on_IC}
In the present section we assume \eqref{a: boundary regularity}, \eqref{a: boundedness} and that $Y$ as in Definition~\ref{def:Y} is such that its flattening $Y_0$ satisfies \eqref{Ellipticity}.  Further, $1 < p < \infty$.

\begin{definition}
\label{def:Dynamical Systems Norm}
 For $t\in [0,T]$, $a \in Y$, $u_0 \in C([-1,0],L_p(D))$ and $R \in \mathcal{R}$ we define
 \begin{align*}
     \delta(t;a,u_0,R) & =\sup_{\vartheta\in[-1,0]}\norm{u(t+\vartheta;a,u_0,R)}_{L_p(D)}
     \\
     & = \norm{u(t+\cdot\,;a,u_0,R)\upharpoonright_{[-1,0]}}_{C([-1,0],L_p(D))}.
 \end{align*}

For notational simplicity we often write $u(t+\,\cdot)$ instead of $u(t+\, \cdot\, ;a,u_0,R)$ and  $\delta(t)$ instead of $\delta(t;a,u_0,R)$ when $a \in Y$ and $u_0 \in C([-1,0],L_p(D))$ are fixed and this does not lead to confusion.
\end{definition}

\begin{lemma}
For any $a \in Y$, $u_0 \in C([-1,0],L_p(D))$ and $R \in \mathcal{R}$ the function $\delta(\cdot;a,u_0,R) \colon [0,T] \to \RR^+$ is continuous.
\end{lemma}
\begin{proof}
First note that the mapping
\begin{equation*}
    \big[\, [0,T] \times [-1,0]\ni (t,\vartheta) \mapsto \norm{u(t+\vartheta;a,u_0,R)}_{L_p(D)} \in \RR^+ \,\big]
\end{equation*}
is continuous as a composition of continuous mappings. Due to the compactness of $[-1,0]$ the $\delta(\cdot,a,u_0,R)$ function is continuous when $a \in Y$, $R\in \mathcal{R}$ and $u_0\in C([-1,0],L_p(D))$ are fixed. \qed
\end{proof}

\begin{lemma}
\label{lemma: mild estim 1}
There are constants $M_1, M_2$ such that for any $\rho \in [0,T]$ the inequality
\begin{equation*}
     \norm{u(\rho;a,u_0,R)}_{L_p(D)} \le M_1 \delta(0;a,u_0,R) + M_2 \int_{0}^{\rho} \delta(\zeta;a,u_0,R) \, \dd\zeta
\end{equation*}
holds  for all $a \in Y$, $u_0\in C([-1,0],L_p(D))$ and $R \in\mathcal{R}$.
\end{lemma}

\begin{proof}
Fix $\rho\in[0,T]$ and note that
\begin{align*}
    \norm{u(\rho)}_{L_p(D)}&\le \norm{U_{\tilde{a}}(\rho)u_0(0)}_{L_p(D)}+\int_{0}^{\rho} \norm{\U_{\tilde{a}} (\rho,\zeta) C^{1}_{a}(\zeta)(u\circ\Phi)(\zeta) }_{L_p(D)}\,\dd\zeta \\
    &\le Me^{\gamma \rho} \norm{u_0(0)}_{L_p(D)}+  Me^{\gamma\rho}K \int_{0}^{\rho} \norm{(u\circ\Phi)(\zeta) }_{L_p(D)}\,\dd\zeta \\
    &\le Me^{\gamma \rho} \norm{u_0}_{C([-1,0],L_p(D))}+  Me^{\gamma\rho}K \int_{0}^{\rho} \norm{u(\zeta+\cdot)\upharpoonright_{[-1,0]}}_{C([-1,0],L_p(D))}\,\dd\zeta \\
    &=M_1 \delta(0)+ M_2\int_{0}^{\rho} \delta(\zeta)\,\dd\zeta,
\end{align*}
where $M$ is a uniform bound of the operator $U_{\tilde{a}}(t)$ with respect to $a \in Y$ and $0 \le t \le T$ (see Proposition~\ref{prop:higher_orders-skew_product-p-q-estimates_2}), the constant $K$ is a uniform bound of the operator $C^{1}_{a}(\zeta)$ with respect to $a \in Y$ and $0 \le t \le T$ (see Lemma~\ref{Boundednes of Multiplication Operator}).  Moreover, the bounds $M$ and $K$ are independent on initial condition $u_0$. By setting $M_1 = Me^{\gamma T}$, $M_2 = MKe^{\gamma T}$ we end the proof.  \qed
\end{proof}

From now on, throughout this section the constants $M_1$ and $M_2$ will be defined as in  Lemma~\ref{lemma: mild estim 1}.

\begin{proposition}
\label{zero_initial_condi_esti}
For any sequence $(u_{0,m})_{m=1}^{\infty}\subset C([-1,0],L_p(D))$ convergent to zero, any $t \in[0,T]$, $R\in \mathcal{R}$ and $a \in Y$ the sequence $\delta_m(t) \vcentcolon= \delta(t;a,u_{0,m},R)$ converges to zero.
\end{proposition}
\begin{proof}
Fix $t \in[0,T]$ and $-1\le \vartheta\le 0$ and let us consider two cases.

\begin{itemize}[label=$\bullet$]
    \item If $0\le t+\vartheta\le T$ then from Lemma~\ref{lemma: mild estim 1} there holds
\begin{align*}
    \norm{u(t+\vartheta)}_{L_p(D)}&\le M_1 \delta(0)+ M_2\int_{0}^{t+\vartheta} \delta(\zeta)\,\dd\zeta \\
    &\le M_1 \delta(0)+ M_2\int_{0}^{t} \delta(\zeta)\,\dd\zeta.
\end{align*}
\item If $-1\le t+\vartheta\le 0$ then the inequality
\begin{align*}
    \norm{u(t+\vartheta)}_{L_p(D)} \le M_1 \delta(0)+ M_2\int_{0}^{t} \delta(\zeta)\,\dd\zeta
\end{align*}
is straightforward, as even the stronger one $\norm{u(t+\vartheta)}_{L_p(D)} \le M_1 \delta(0)$ is true.
\end{itemize}

Applying $\sup$ with respect to $\vartheta$ on both sides give us that
\begin{equation*}
    \sup_{\vartheta\in[-1,0]}\norm{u(t+\vartheta,a,u_0,R)}_{L_p(D)}\le M_1 \delta(0)+ M_2\int_{0}^{t} \delta(\zeta)\,\dd\zeta,
\end{equation*}
what can be rewritten in terms of the $\delta$ function as
\begin{equation*}
   \delta(t)\le M_1 \delta(0)+ M_2\int_{0}^{t} \delta(\zeta)\,\dd\zeta.
\end{equation*}
The function $\delta$ is nonnegative and continuous on the compact domain, hence it is integrable. Using the Grönwall lemma we get
\begin{equation}
\label{eq:Gronwall}
    \delta(t)\le M_1 \delta(0)\exp \Big(M_2\int_{0}^{t}\,\dd\zeta \Big).
\end{equation} \qed
\end{proof}

The above Lemma~\ref{lemma: mild estim 1} and Proposition~\ref{zero_initial_condi_esti} lead to global $L_p$\nobreakdash-\hspace{0pt}norm estimation of the mild solution of~$\eqref{main-eq}_a+\eqref{main-bc}_a$ in terms of initial conditions.
\begin{proposition}
\label{prop: global mild L_p estim}
There is constant $\overline{M} > 0$ such that inequality
\begin{equation*}
     \norm{u(t;a,u_0,R)}_{L_p(D)} \le \overline{M} \norm{u_0}_{C([-1,0],L_p(D))}
\end{equation*}
holds for any $1<p<\infty$, $t\in [0,T]$, $a\in Y$, $R\in\mathcal{R}$ and $u_0 \in C([-1,0],L_p(D))$.
\end{proposition}
\begin{proof}
Let $\overline{M} = M_1\exp(M_2T)$, where $M_1,M_2$ are constants as in Lemma~\ref{lemma: mild estim 1}. Then by Proposition~\ref{zero_initial_condi_esti} we can write
\begin{align*}
     \norm{u(t;a,u_0,R)}_{L_p(D)} \le {} & \delta(t;a,u_0,R)\\
     \le {} & \overline{M} \norm{u_0}_{C([-1,0],L_p(D))}.
\end{align*} \qed
\end{proof}
\begin{proposition}
 For any $a\in Y$ and $R\in\mathcal{R}$ the mapping
\begin{equation*}
    \big[\,C([-1,0],L_p(D))\ni u_0\mapsto u(\cdot;a,u_0,R) \in C([-1,T],L_p(D)) \,\big]
\end{equation*}
is continuous.
\end{proposition}
\begin{proof}
Let $a\in Y$, $R\in\mathcal{R}$ be fixed. Then in the spirit of Cauchy's definition we can find that
\begin{multline*}
        \norm{u(\cdot;a,u_{0,1},R)- u(\cdot;a,u_{0,2},R)}_{C([-1,T],L_p(D))}
        \\
        \le \sup_{t\in[0,T]}\delta(t;a,u_{0,1}-u_{0,2},R)
        \\
        \le M_1\exp (M_2T) \norm{u_{0,1}-u_{0,2}}_{C([-1,0],L_p(D))}
\end{multline*}
for any initial conditions $u_{0,1}, u_{0,2}\in C([-1,0],L_p(D))$. The first inequality results from the  linearity of the problem $\eqref{main-eq}_{a}+\eqref{main-bc}_a$ and the second inequality follows from~\eqref{eq:Gronwall}. \qed
\end{proof}

\section{Continuous Dependence on Coefficients and Delay}
\label{sect:dependence-on-parameters}
In the present section we assume \eqref{a: boundary regularity}, \eqref{a: boundedness} and that $Y$ as in Definition~\ref{def:Y} is such that its flattening $Y_0$ satisfies \eqref{Ellipticity} and \eqref{a: compactness of Y }.  As in Section~\ref{sect:dependence_on_IC}, $1 < p < \infty$.

\begin{proposition}
\label{lm:continuity-wrt-parameters-1}
 For any $0 < T_1 \le T $, $R \in \mathcal{R}$ and $u_0 \in C([- 1, 0], L_p(D))$ the mapping
    \begin{equation*}
        \bigl[\, Y\ni a \mapsto u(\cdot; a, u_0,R)\!\!\restriction_{[T_1, T]}  \in C([T_1, T], L_p(D)) \,\bigr]
    \end{equation*}
    is continuous.
\end{proposition}
\begin{proof}
Fix $p$, $T_1$, $R$ and $u_0$ as in the Proposition.  Let $(a_{m})_{m=1}^{\infty} \subset Y$ converge to $a$. Put $u_m(\cdot)$ for $u(\cdot; a_m, u_0,R)$ and $u(\cdot)$ for $u(\cdot; a, u_0,R)$. It suffices to prove that there is a subsequence $(a_{m_k})_{k=1}^{\infty} \subset Y$ such that $u_{m_k}(\cdot)$ converges to $u(\cdot)$ on $[T_1,T]$ uniformly. By Theorem~\ref{thm:delay-compactness-X_2} via diagonal process, we can find a subsequence $u_{m_k}$ such that $u_{m_k}\!\!\restriction_{(0, T]}$ converge to some continuous $\hat{u} \colon (0, T] \to L_{p}(D)$ and the convergence is uniform on compact subsets of $(0, T]$. The function $\hat{u}$ is clearly bounded by Proposition~\ref{prop: global mild L_p estim}. Moreover, we extend the map $\hat{u}$ to the whole $[-1,T]$ by $u_0$ on $[-1,0]$, i.e., now
\begin{equation*}
\tilde{u}(t) \vcentcolon=\left\{\begin{array}{lll}
u_0(t) & \quad \text{for } t\in [-1,0] \\[1ex]
\lim\limits_{k \to \infty} u_{m_k}(t) & \quad \text{for } t\in (0,T].
\end{array}\right.
\end{equation*}
It remains to prove that $\tilde{u} = u$. In order not to overburden the notation we write $u_m$ instead of $u_{m_k}$.

Our first step is to show that, for each $t \in (0, T]$,
    \begin{align}
    \label{cont-1-Y}
       U_{\tilde{a}_{m}}(t, 0) u_{0}(0) & \to U_{\tilde{a}}(t, 0) u_{0}(0)
        \\
    \label{cont-2-Y}
        \int\limits_{0}^{t} U_{\tilde{a}_m}(t,\zeta)C^1_{a_m}(\zeta)u_m(\zeta-R(\zeta)) \, \mathrm{d}\zeta & \to \int\limits_{0}^{t} U_{\tilde{a}}(t,\zeta)C^1_{a}(\zeta)\tilde{u}(\zeta-R(\zeta)) \, \mathrm{d}\zeta.
    \end{align}
in the $L_{p}(D)$\nobreakdash-\hspace{0pt}norm as $m \to \infty$. The convergence in \eqref{cont-1-Y} is a consequence of Proposition~\ref{prop:higher_orders-skew_product-continuity}.  The convergence in \eqref{cont-2-Y} can be shown by showing the convergence of the difference
    \begin{equation}
    \label{J_1_2_3_Y}
        \begin{aligned}
             \int\limits_{0}^{t} U_{\tilde{a}_{m}}(t, \zeta) & C^1_{a_{m}}(\zeta)  u_{m}(\zeta-R(\zeta)) \, \dd\zeta  -  \int\limits_{0}^{t} U_{\tilde{a}}(t, \zeta) C^1_{a}(\zeta) \tilde{u}(\zeta-R(\zeta))\, \dd\zeta\\
            =& \int\limits_{0}^{t} (U_{\tilde{a}_{m}}(t, \zeta) - U_{\tilde{a}}(t, \zeta) ) C^1_{a_{m}}(\zeta) u_m(\zeta-R(\zeta)) \, \dd\zeta\\
            &+\int\limits_{0}^{t}  U_{\tilde{a}}(t, \zeta)  C^1_{a_{m}}(\zeta) (u_{m}(\zeta-R(\zeta))-\tilde{u}(\zeta-R(\zeta))) \, \dd\zeta\\
            &+\int\limits_{0}^{t}  U_{\tilde{a}}(t, \zeta)  (C^1_{a_{m}}(\zeta)-C^1_{a}(\zeta)) \tilde{u}(\zeta-R(\zeta)) \, \dd\zeta
        \end{aligned}
    \end{equation}
to zero.  Write $J^{(i)}_{m}(t)$, $i = 1, 2, 3$, for the $i$\nobreakdash-\hspace{0pt}th term on the right-hand side of~\eqref{J_1_2_3_Y}.  The convergence of $J^{(1)}_{m}(t)$ follows from the Lebesgue dominated convergence theorem for Bochner integral: since the integrand
  \begin{equation*}
      (0,t)\ni \zeta \mapsto (U_{\tilde{a}_{m}}(t, \zeta) - U_{\tilde{a}}(t, \zeta) ) C^1_{a_{m}}(\zeta) u_m(\zeta-R(\zeta)) \in L_p(D)
  \end{equation*}
is  $(\mathfrak{L}((0, t)), \mathfrak{B}(L_p(D)))$\nobreakdash-\hspace{0pt}measurable for all $m\in \NN$ (see Lemmas~\ref{lm:I_i_(i)} and \ref{lemma:v_m-boundednes}) and bounded uniformly (see Proposition~\ref{prop: global mild L_p estim}) in $m\in \NN$:
\begin{multline*}
            \norm{ (U_{\tilde{a}_{m}}(t, \zeta) - U_{\tilde{a}}(t, \zeta) ) C^1_{a_{m}}(\zeta) u_m(\zeta-R(\zeta)) }_{L_p(D)} \\
            \le  2Me^{\gamma T}K\overline{M}\norm{u_0}_{C([-1,0],L_p(D))}
\end{multline*}
it suffices to check that for a.e.\ $\zeta\in (0,t)$ the integrand converges to zero, which follows from the estimate
\begin{multline*}
            \norm{ (U_{\tilde{a}_{m}}(t, \zeta) - U_{\tilde{a}}(t, \zeta) ) C^1_{a_{m}}(\zeta) u_m(\zeta-R(\zeta)) }_{L_p(D)} \\
            \le    \norm{ (U_{\tilde{a}_{m}}(t, \zeta) - U_{\tilde{a}}(t, \zeta) ) }_{\mathcal{L}(L_p(D))}K \overline{M}\norm{u_0}_{C([-1,0],L_p(D))}
\end{multline*}
and Proposition~\ref{prop:higher_orders-skew_product-norm_continuity}.

In order to prove $J_m^{(2)}(t)\to 0$ as $m\to \infty$ we proceed similarly. We see that the mapping
    \begin{equation*}
       \big[\, (0,t)\ni \zeta\mapsto U_{\tilde{a}}(t, \zeta)   C^1_{a_{m}}(\zeta) (u_{m}(\zeta-R(\zeta))-\tilde{u}(\zeta-R(\zeta))) \in L_p(D)\, \big]
    \end{equation*}
is $(\mathfrak{L}((0, t)), \mathfrak{B}(L_p(D)))$\nobreakdash-\hspace{0pt}measurable for all $m \in \NN$, as a consequence of Lemmas~\ref{lm:I_i_(i)}  and \ref{lemma:v_m-boundednes}, and bounded uniformly in $m\in \NN$, since, by Proposition~\ref{prop:higher_orders-skew_product-p-q-estimates_2}, Lemma~\ref{Boundednes of Multiplication Operator} and Proposition~\ref{prop: global mild L_p estim},
\begin{multline*}
            \norm{U_{\tilde{a}}(t, \zeta)   C^1_{a_{m}}(\zeta) (u_{m}(\zeta-R(\zeta))-\tilde{u}(\zeta-R(\zeta))) }_{L_p(D)} \\
            \le  2MKe^{\gamma T}\overline{M}\norm{u_0}_{C([-1,0],L_p(D))}.
\end{multline*}
Further, the convergence, for a.e.\ $\zeta\in (0,t)$,
        \begin{equation*}
            U_{\tilde{a}}(t, \zeta)   C^1_{a_{m}}(\zeta) (u_{m}(\zeta-R(\zeta))-\tilde{u}(\zeta-R(\zeta))) \to 0
        \end{equation*}
in $L_p(D)$ is due to the pointwise convergence of $u_m$ to $\tilde{u}$ on $[-1,T]$ and the estimate (by Proposition~\ref{prop:higher_orders-skew_product-p-q-estimates_2} and Lemma~\ref{Boundednes of Multiplication Operator})
 \begin{multline*}
             \norm{U_{\tilde{a}}(t, \zeta)   C^1_{a_{m}}(\zeta) (u_{m}(\zeta-R(\zeta))-\tilde{u}(\zeta-R(\zeta))) }_{L_p(D)}\\ \le
              MKe^{\gamma T} \norm{ u_{m}(\zeta-R(\zeta))-\tilde{u}(\zeta-R(\zeta)) }_{L_p(D)}.
        \end{multline*}

The convergence of $J^{(3)}_{m}(t)$ follows from the facts that the set
    \begin{equation*}
        \{\, J^{(3)}_{m}(t) : m \in \NN \,\}
    \end{equation*}
is precompact in $L_{p}(D)$ (see Lemma~\ref{lm:compact-aux}, Lemma~\ref{lemma:v_m-boundednes}, Proposition~\ref{prop: global mild L_p estim})
and that $J^{(3)}_{m}(t)$  converge weakly to zero, i.e.,
    \begin{align}
    \label{cont-3-3_Y}
        \langle J^{(3)}_{m}(t), v \rangle & \to 0 \text{ as } m \to \infty,
    \end{align}
for any $v \in L_{{p'}}(D)$, where $\langle \cdot, \cdot \rangle$ stands for the duality pairing between $L_{p}(D)$ and $L_{p'}(D)$.  By the Hille theorem (\cite[Thm.~II.2.6]{DiUhl}) and Proposition~\ref{prop:dual-Lp-q},
        \begin{multline*}
        \langle J^{(3)}_{m}(t), v \rangle =  \int_{0}^{t} \langle U_{\tilde{a}}(t, \zeta)  (C^1_{a_{m}}(\zeta)-C^1_{a}(\zeta)) \tilde{u}(\zeta-R(\zeta)) , v \rangle \, \dd\zeta
        \\
        = \int_{0}^{t} \langle  (C^1_{a_{m}}(\zeta)-C^1_{a}(\zeta)) \tilde{u}(\zeta-R(\zeta)) , U^{*}_{\tilde{a}}(t, \zeta)v  \rangle \, \dd\zeta.
    \end{multline*}
Now we need to use a subtler approach based on the convergence $c_{1,m}$ to $c_{1}$ in the \WST topology of $L_{\infty}((0,t) \times D)$. First note that the mappings
\begin{gather*}
        \big[\, (0,t)\ni \zeta\mapsto \tilde{u}(\zeta-R(\zeta)) \in L_p(D) \,\big] \quad \& \quad
    \big[\, (0,t)\ni\zeta \mapsto U^{*}_{\tilde{a}}(t, \zeta)v \in L_{p'}(D) \,\big]
\end{gather*}
belong to $L_{\infty}((0,t),L_p(D))$ and $L_{\infty}((0,t),L_{p'}(D))$ respectively. Therefore the mapping (the product of the above maps)
\begin{equation*}
     \big[\, (0,t)\times D \ni (\zeta,x) \mapsto \tilde{u}(\zeta-R(\zeta))[x] (U^{*}_{\tilde{a}}(t, \zeta)v)[x] \in\RR \,\big]
\end{equation*}
belong to $L_{1}((0,t)\times D)$ see~Lemma~\ref{DS_a}. It suffices now to note that from Fubini's theorem we have
\begin{multline*}
               \int_0^t  \langle  (C^1_{a_{m}}(\zeta)-C^1_{a}(\zeta))\tilde{u}(\zeta-R(\zeta)) , U^{*}_{\tilde{a}}(t, \zeta)v \rangle\, \mathrm{d}\zeta \\= \int_{0}^t \!\! \int_{D} (c_{1,m}(\zeta, x) - c_1(\zeta, x)) \, \tilde{u}(\zeta-R(\zeta))[x] \, (U^*_{\tilde{a}}(\zeta, t) v)[x] \, \mathrm{d}x  \, \mathrm{d}\zeta,
\end{multline*}
so the integral tends to zero as $m\to \infty$.

We have thus proved that
    \begin{align*}
        \tilde{u}(t) = {} &  \U_{\tilde{a}} (t,0)u_0(0) +\int_{0}^{t} \U_{\tilde{a}} (t,\zeta) C^{1}_{a}(\zeta)\tilde{u}(\zeta-R(\zeta)) \,\dd\zeta, \quad t \in [0, T].
    \end{align*}

Now we prove the continuity of the extension $\tilde{u}$. Note that the only point where it can fail is $t=0$. However, this is not the case since the mappings
    \begin{gather*}   \Big[\, [0,T]\ni  t\mapsto    \U_{\tilde{a}} (t,0)u_0(0) \in L_p(D) \,\Big] \\
              \Big[\, [0,T]\ni  t\mapsto   \int_{0}^{t} \U_{\tilde{a}} (t,\zeta)  C^{1}_{a}(\zeta)\tilde{u}(\zeta-R(\zeta)) \,\dd\zeta\in L_p(D) \,\Big]
    \end{gather*}
    are continuous (see Lemmas~\ref{prop:higher_orders-skew_product-p_ii} and~\ref{lm:I_i_(ii)}), so the mapping $\tilde{u}$ is continuous on the whole $[-1,T]$.  Also, $\tilde{u} = u_0$ on $[-1,0]$, hence $\tilde{u}$ is in~fact the mild solution of $\eqref{main-eq}_{\tilde{a}}+\eqref{main-bc}_{\tilde{a}}$, therefore, by uniqueness, $\tilde{u}(t) = u(t;\tilde{a},u_0,R)$ for any $t \in [-1, T]$.  \qed
\end{proof}

\begin{proposition}
\label{prop:continuity-wrt-parameters-tau}
Assume additionally~\eqref{a:only_c_continuous}. For any $u_0 \in C([- 1, 0], L_p(D))$ and $\mathcal{R}_0\subset\mathcal{R}$ satisfying the assumption~\eqref{a:tau_pointwise_convergence} the mapping
    \begin{equation*}
        \bigl[\, Y\times \mathcal{R}_0 \ni (a,R) \mapsto u(\cdot; a, u_0,R)\!\!\restriction_{[0, T]}  \in C([0, T], L_p(D)) \,\bigr]
    \end{equation*}
    is continuous.
\end{proposition}
\begin{proof}[Sketch of proof]
Fix $p$, $u_0$ and $\mathcal{R}_0$. Let $(a_{m})_{m=1}^{\infty} \subset Y$ converge to $a$ and $(R_m)_{m=1}^{\infty}\subset \mathcal{R}_0$ converge to $R$, and put $u_m(\cdot)$ for $u(\cdot; a_m, u_0,R_m)$ and $u(\cdot)$ for $u(\cdot; a, u_0,R)$. We will proceed as in Proposition~\ref{lm:continuity-wrt-parameters-1}.  In~particular, $\tilde{u}$ has the same meaning.  However, we have in~fact more:  as we assume~\eqref{a:only_c_continuous}, we can apply Theorem~\ref{c:delay-compactness-up_to_0} to show that $u_m$ converge to $\hat{u}$ uniformly on $[0, T]$, from which it follows in~particular that $\tilde{u}$ is continuous on the whole of $[-1, T]$.

It follows again from \eqref{a:only_c_continuous} that $\tilde{a}_m = \tilde{a}$ for all $m \in \NN$, so
\begin{equation*}
    U_{\tilde{a}_{m}}(t, 0) u_{0}(0) \to U_{\tilde{a}}(t, 0) u_{0}(0)
\end{equation*}
holds trivially.  We start by showing that, for each $t \in [0, T]$,
    \begin{equation}
    \label{cont-2+DA6,7}
        \int\limits_{0}^{t} U_{\tilde{a}}(t,\zeta)C^1_{a_m}(\zeta)u_m(\zeta-R_m(\zeta)) \, \mathrm{d}\zeta \to \int\limits_{0}^{t} U_{\tilde{a}}(t,\zeta)C^1_{a}(\zeta)\tilde{u}(\zeta-R(\zeta)) \, \mathrm{d}\zeta.
    \end{equation}
in the $L_{p}(D)$\nobreakdash-\hspace{0pt}norm as $m \to \infty$. The above convergence can be proved by showing convergence of the terms
    \begin{equation}
    \label{J_1_2+DA6,7}
        \begin{aligned}
             \int\limits_{0}^{t} U_{\tilde{a}}(t, \zeta) & C^1_{a_{m}}(\zeta)  u_m(\zeta-R_m(\zeta)) \, \dd\zeta  -  \int\limits_{0}^{t} U_{\tilde{a}}(t, \zeta) C^1_{a}(\zeta) \tilde{u}(\zeta-R(\zeta))\, \dd\zeta\\
            &=\int\limits_{0}^{t}  U_{\tilde{a}}(t, \zeta)  C^1_{a_{m}}(\zeta) (u_m(\zeta-R_m(\zeta))-\tilde{u}(\zeta-R(\zeta))) \, \dd\zeta\\
            &+\int\limits_{0}^{t}  U_{\tilde{a}}(t, \zeta)  (C^1_{a_{m}}(\zeta)-C^1_{a}(\zeta)) \tilde{u}(\zeta-R(\zeta)) \, \dd\zeta.
        \end{aligned}
    \end{equation}
to zero.  Write $K^{(i)}_{m}(t)$, $i = 1, 2$, for the $i$\nobreakdash-\hspace{0pt}th term on the right-hand side of~\eqref{J_1_2+DA6,7}.

Regarding the convergence of $K^{(1)}_{m}(t)$ to zero, we show the $(\mathfrak{L}((0, t)), \mathfrak{B}(L_p(D)))$\nobreakdash-\hspace{0pt}measurability, for all $m \in \NN$, of the integrand
\begin{equation*}
    \bigl[\, (0,t)\ni \zeta\mapsto U_{\tilde{a}}(t, \zeta)   C^1_{a_{m}}(\zeta) (u_m(\zeta-R_m(\zeta))-\tilde{u}(\zeta-R(\zeta))) \in L_p(D)\, \bigr]
\end{equation*}
in the same way as in the proof of the convergence of $J^{(2)}_{m}(t)$ in  Proposition~\ref{lm:continuity-wrt-parameters-1}.  The fact that for a.e.\ $\zeta\in (0,T)$ we have that
        \begin{equation*}
            U_{\tilde{a}}(t, \zeta)   C^1_{a_{m}}(\zeta) (u_m(\zeta-R_m(\zeta))-\tilde{u}(\zeta-R(\zeta))) \to 0
        \end{equation*}
in $L_p(D)$ is due, in view of ~\eqref{a:tau_pointwise_convergence}, to the uniform convergence of $u_m$ to $\tilde{u}$ on $[-1, T]$ together with the estimate
\begin{multline*}
             \norm{U_{\tilde{a}}(t, \zeta)   C^1_{a_{m}}(\zeta) (u_m(\zeta-R_m(\zeta))-\tilde{u}(\zeta-R(\zeta))) }_{L_p(D)}\\ \le
              MKe^{\gamma T} \norm{ u_m(\zeta-R_m(\zeta))-\tilde{u}(\zeta-R(\zeta)) }_{L_p(D)}.
\end{multline*}
The proof of the convergence of $K^{(2)}_{m}(t)$ to zero is just a copy, word~for~word, of the proof of the convergence of $J^{(3)}_{m}(t)$ in Proposition~\ref{lm:continuity-wrt-parameters-1}. \qed
\end{proof}

\begin{theorem}~
\label{th:joint_continouity}
\begin{enumerate}[label=\textup{(\roman*)},ref=\textup{(\roman*)}]
\item\label{th:joint_continouity_i}  For any $0 < T_1 \le T$ and $R \in \mathcal{R}$ the mapping
    \begin{multline*}
        \Bigl[\, Y \times C([- 1, 0], L_p(D)) \ni (a, u_0)
        \\
        \mapsto u(\cdot; a, u_0,R)\!\!\restriction_{[T_1, T]} \in C([T_1, T], L_{p}(D)) \,\Bigr]
    \end{multline*}
    is continuous.
\item\label{th:joint_continouity_ii}
Under \eqref{a:only_c_continuous}, if $\mathcal{R}_0\subset\mathcal{R}$ is a subset such that the assumption~\eqref{a:tau_pointwise_convergence} holds then the mapping
    \begin{multline*}
        \Bigl[\, Y \times C([- 1, 0], L_p(D))\times \mathcal{R}_0 \ni (a, u_0, R)
        \\
        \mapsto u(\cdot; a, u_0,R)\!\!\restriction_{[0, T]} \in C([0, T], L_{p}(D)) \,\Bigr]
    \end{multline*}
    is continuous.
\end{enumerate}
\end{theorem}
\begin{proof}
Fix $1 < p < \infty$. We start by proving~\ref{th:joint_continouity_i}, so fix also $T_1$, $R$ as in the statement.  Let a sequence $(a_m)_{m=1}^{\infty}\subset Y$ converge to $a \in Y$ and $(u_{0,m})_{m=1}^{\infty}\subset C([-1,0],L_p(D))$ converge to $u_0\in C([-1,0],L_p(D))$. The main idea of the proof is based on the estimation
\begin{equation}
\label{eq:joint_continouity}
    \begin{multlined}
         \norm{u(\cdot; a_m, u_{0,m}, R){\restriction}_{[T_1, T]}  - u(\cdot; a, u_{0}, R){\restriction}_{[T_1, T]} }_{C([T_1, T], L_{p}(D))}
        \\
        \quad \le \norm{ u(\cdot; a_m, u_{0,m}, R){\restriction}_{[T_1, T]} - u(\cdot; a_m, u_{0}, R) {\restriction}_{[T_1, T]}}_{C([T_1, T], L_{p}(D))}
        \\
         + \norm{u(\cdot; a_m, u_0,R){\restriction}_{[T_1, T]} - u(\cdot; a, u_0,R){\restriction}_{[T_1, T]}}_{C([T_1, T], L_{p}(D))}.
    \end{multlined}
\end{equation}
\noindent
Proposition~\ref{prop: global mild L_p estim} implies
\begin{multline*}
     \norm{ u(\cdot; a_m, u_{0,m}, R){\restriction}_{[T_1, T]} - u(\cdot; a_m, u_{0}, R) {\restriction}_{[T_1, T]}}_{C([T_1, T], L_{p}(D))} \\ \le \overline{M}\norm{u_{0,m}-u_0}_{C([-1,0],L_p(D))}.
\end{multline*}
Therefore the first part of the right-hand side of~\eqref{eq:joint_continouity} converges to zero as $m \to \infty$. The second part of~\eqref{eq:joint_continouity} converges to zero by Proposition~\ref{lm:continuity-wrt-parameters-1}.  Item~\ref{th:joint_continouity_ii} can be proved similarly.  So, assume additionally \eqref{a:only_c_continuous} and, instead of fixing the delay $R\in \mathcal{R}$ take a sequence $(R_m)_{m=1}^{\infty} \subset \mathcal{R}_0$ convergent to some $R\in \mathcal{R}_0$.  By similar estimation,
\begin{equation}
\label{eq:joint_continouity_R}
    \begin{multlined}
    \norm{u(\cdot; a_m, u_{0,m}, R_m){\restriction}_{[0, T]} - u(\cdot; a, u_{0}, R){\restriction}_{[0, T]} }_{C([0, T], L_{p}(D))}\\
        \quad \le  \norm{ u(\cdot; a_m, u_{0,m}, R_m){\restriction}_{[0, T]} - u(\cdot; a_m, u_{0}, R_m) {\restriction}_{[0, T]}}_{C([0, T], L_{p}(D))}
        \\
       + \norm{u(\cdot; a_m, u_0,R_m){\restriction}_{[0, T]} - u(\cdot; a, u_0,R){\restriction}_{[0, T]}}_{C([0, T], L_{p}(D))}
    \end{multlined}
\end{equation}
together with Propositions~\ref{prop: global mild L_p estim} and \ref{prop:continuity-wrt-parameters-tau} concludes the proof.   \qed
\end{proof}

\end{document}